%% file: cubicosc.tex
\documentclass[a4paper,11pt]{article}
\usepackage{amsmath}
\usepackage{amssymb}
\usepackage{amsthm}
\usepackage{authblk}
\usepackage{comment}
\usepackage[english]{babel}
\usepackage{a4wide}
\usepackage{overpic}
\usepackage{graphicx}
\usepackage{tikz}
\usepackage{pgfplots}
\usepackage{subfigure}
\usepackage{url}
\usepackage{hyperref}
\hypersetup{
    colorlinks,%
    citecolor=black,
    filecolor=black,
    linkcolor=black,
    urlcolor=black
}

\newlength{\figurewidth}
\newlength{\figureheight}

\usepackage{color}

\newtheorem{theorem}{Theorem}[section]
\newtheorem{proposition}[theorem]{Proposition}

\newtheorem{lemma}[theorem]{Lemma}

\newtheorem{remark}[theorem]{Remark}

\numberwithin{equation}{section}
    \newcommand{\Ai}{\mathrm{Ai}}

\author[1]{Daan Huybrechs}
\author[2]{Arno B.J. Kuijlaars}
\author[1]{Nele Lejon} 
\affil[1]{KU Leuven, Department of Computer Science,  Celestijnenlaan 200A, 3001 Leuven, Belgium, daan.huybrechs@cs.kuleuven.be, nele.lejon@gmail.com}
\affil[2]{KU Leuven, Department of Mathematics,  Celestijnenlaan 200B, 3001 Leuven, Belgium, arno.kuijlaars@kuleuven.be}

\date{}
\title{A numerical method for oscillatory integrals with \\coalescing saddle points}
\begin{document}
\maketitle
\begin{abstract}
The value of a highly oscillatory integral is typically determined asymptotically by the behaviour of the integrand near a small number of critical points. These include the endpoints of the integration domain and the so-called stationary points or saddle points -- roots of the derivative of the phase of the integrand -- where the integrand is locally non-oscillatory. Modern methods for highly oscillatory quadrature exhibit numerical issues when two such saddle points coalesce. On the other hand, integrals with coalescing saddle points are a classical topic in asymptotic analysis, where they give rise to uniform asymptotic expansions in terms of the Airy function. In this paper we construct Gaussian quadrature rules that remain uniformly accurate when two saddle points coalesce. These rules are based on orthogonal polynomials in the complex 
plane. We analyze these polynomials, prove their 
existence for even degrees, and describe an accurate and efficient numerical scheme for the evaluation of oscillatory integrals with coalescing saddle points.
\end{abstract}

\section{Introduction and statement of results}
\label{sect_intro}

\subsection{Introduction}

Highly oscillatory integrals are a challenge for numerical integration methods, as the oscillatory nature of an integrand typically necessitates a large number of quadrature points. However, efficient numerical methods have been described for oscillatory integrals of the form
\begin{equation}\label{eq:I}
 I[f]=\int_a^b f(x)e^{i\omega g(x)} {\rm d}x,
\end{equation}
where both $f$ and $g$ are smooth functions and the oscillations can be attributed to a large value of the frequency parameter $\omega$. Examples include Filon-type quadrature, Levin quadrature, the numerical method of steepest descent and others \cite{iserles2005quad,olver2006levin,huybrechs2006osc1,huybrechs2009hoq,deano2018book}. Integrals of this form, and variations thereof, arise in a variety of applications typically involving wave phenomena.

An advantageous property of integral \eqref{eq:I} is that it is amenable to asymptotic analysis, and its value can be approximated using the method of stationary phase or the method of steepest descent \cite{bleistein1975asymptotic,wong2001asymptotic}. This approximation, being asymptotic, improves with increasing $\omega$. This is unlike classical quadrature schemes, which deteriorate with increasing $\omega$. The main goal of highly oscillatory quadrature methods, such as the methods mentioned above, is to combine improved accuracy for increasing $\omega$ with numerical convergence for any value of $\omega$, at a computational cost that is independent of $\omega$.

Asymptotic analysis of oscillatory integrals becomes more involved in the presence of \emph{stationary points} or \emph{saddle points}. These are roots of the derivative of the phase function, $g'(\xi)=0$, around which the integrand is locally non-oscillatory. Highly oscillatory quadrature techniques have to explicitly take such stationary points into account. The situation worsens when two stationary points are close to each other or when two stationary points coalesce for a particular value of a parameter. A canonical example is given by the cubic oscillator with a linear perturbation,
\begin{equation}\label{eq:osc}
 g(x,c)=\frac{x^3}{3}-c x.
\end{equation}
This oscillator features stationary points at $\pm \sqrt{c}$, coalescing at $c=0$. These stationary points are on the real line when $c>0$, and on the imaginary axis if $c<0$.

Classical Poincar{\'e}-type asymptotic expansions (i.e., using just integer powers of $\omega^{-1})$ break down in the presence of two coalescing saddle points. This problem is well-known in asymptotic analysis and the solution is to consider uniform asymptotic expansions, see for example \cite{olver1974specialfunctions,bleistein1975asymptotic,wong2001asymptotic}. They are uniform in the sense that they are valid for a range of the parameter $c$, including its critical value $c=0$. Uniform asymptotic expansions typically involve a special function that captures the special transitional behaviour around a critical value of a parameter. In the case of two coalescing saddle points, the special function is the classical Airy function, which itself has an integral representation that involves a cubic oscillator. An example of such an expansion in this paper is \eqref{eq:airy_expansion} further on.

Numerical methods based on the existence of asymptotic expansions do not necessarily break down completely, but they certainly deteriorate in the presence of coalescing saddle points. In this paper we explore the analogue of uniform asymptotic expansions for one scheme, the numerical method of steepest descent. Though a great variety of uniform asymptotic expansions have been described in the literature for variations of integral \eqref{eq:I}, the Airy case is a canonical example, which has received the most study. For that reason, we pursue this case in detail in this paper.

The goal of this paper is the construction and analysis of a uniformly applicable quadrature rule, uniform in the parameter $c$ near $c=0$, for the canonical integral
\begin{equation}\label{eq:quadrature_rule}
 \int_{-1}^1 f(x) e^{i \omega \left( \frac{x^3}{3}-cx \right)} {\rm d}x \approx \sum_{k=1}^n w_k f(x_k),
\end{equation}
where $n$ is small and independent of $\omega$. We focus on a quadrature rule with optimal asymptotic order, in the sense that the error decays at the fastest algebraic rate in $\omega^{-1}$ among all quadrature rules with $n$ points. A consequence is that the points and weights depend on $\omega$ and $c$ and, furthermore, that the points $x_k$ typically lie in the complex plane. For that reason, we assume that $f$ is an analytic function at least in an open neighbourhood of $[-1,1]$ in the complex plane. Optimal quadrature rules involving real quadrature points in the interval $[-1,1]$ only are the subject of ongoing research, but in any case their convergence for large $\omega$ is slower than that of rules with complex points.

Integral \eqref{eq:quadrature_rule} is just one specific example of an oscillatory integral with coalescing saddle points. The quadrature rule can also be applied to integrals with more general oscillators of the form $e^{i \omega g(x)}$ where $g'$ has two nearby roots, on the real axis or elsewhere in the complex plane. Much like the derivation of uniform asymptotic expansions for such integrals, this requires a smooth change of variables to the canonical case. This change of variables is standard and is recalled in \S\ref{s:genericphase}. An example of the approach is included in \S\ref{sect_application}. Phase functions with a larger number of stationary points would require separate (and more involved) treatment. Other cases that require uniform asymptotics might be amenable to similar numerical techniques, but are not explored in this paper. Each case, such as the coalescence of a stationary point with an endpoint or of a stationary point with a pole, would require its own separate analysis. Uniform asymptotic expansions are known for a sizable number of different cases. An overview, along with the relevant special function in each case, is given in \cite[Chapter 20]{temme2014}.

\subsection{Main results and outline of the paper}

The analysis in this paper centers around a family of orthogonal polynomials $p_{n,\delta}(z)$. They are monic and orthogonal with respect to a complex-valued oscillatory weight function,
\begin{equation}\label{eq:orthogonality}
 \int_\Gamma p_{n,\delta}(z) z^k e^{i \left( \frac{z^3}{3} - \delta z \right)} {\rm d}z=0,\qquad k=0,\ldots, n-1.
\end{equation}
Here, $\Gamma$ is any contour in the complex plane that connects the points $\infty \times e^{\frac{5i\pi}{6}}$ and $\infty \times e^{\frac{i\pi}{6}}$ at infinity. After a suitable rescaling, that will be detailed further on, the roots of these polynomials in combination with the roots of some other (known) polynomials give rise to the sought Gaussian quadrature rule of the form \eqref{eq:quadrature_rule}. The polynomials and their roots are independent of the precise choice of $\Gamma$ in \eqref{eq:orthogonality}, as long as $\Gamma$ connects the two given points at infinity, since the path of integration can be analytically deformed without changing the value of the integral.

The orthogonality conditions \eqref{eq:orthogonality} represent a non-classical setting of orthogonal polynomials, because the weight function is oscillatory. Hence, unique existence of the polynomials is not guaranteed for each value of $\delta$. We proceed by analyzing the corresponding Hankel determinants
\[
h_n = \det H_n,
\]
where $H_n$ is the Hankel matrix given by
\begin{equation}\label{eq:hankel}
  H_n=\left[
  \begin{array}{cccc}
         \mu_0 & \mu_1 & \cdots & \mu_{n-1}\\
         \mu_1 & \mu_2 & \cdots & \mu_{n}\\
         \vdots & \vdots & & \vdots\\
         \mu_{n-1} & \mu_{n} & \cdots & \mu_{2n-2}
  \end{array}
  \right]
\end{equation}
in terms of the \emph{moments} $\mu_k$ of the weight function,
\begin{equation} \label{eq:moments}
 \mu_k = \int_\Gamma z^k e^{i \left( \frac{z^3}{3} - \delta z \right)} {\rm d}z.
\end{equation}
These quantities are all functions of $\delta$, which we omit in our notation. Existence of the polynomial $p_{n,\delta}(z)$ for a particular value of $\delta$ is equivalent to the non-vanishing of $h_n$, since the latter appears in the denominator of the well-known determinant formula for orthogonal polynomials that remains valid in our setting,
\begin{equation}\label{eq:pn_determinant}
  p_{n,\delta}(x)=\frac{1}{h_n}\det\!
  \left[
  \begin{array}{ccccc}
         \mu_0 & \mu_1 & \cdots & \mu_{n-1} & 1\\
         \mu_1 & \mu_2 & \cdots & \mu_n & x\\
         \vdots & \vdots & & \vdots & \vdots\\
         \mu_n & \mu_{n+1} & \cdots & \mu_{2n-1} & x^n
  \end{array}
  \right].
\end{equation}

In the theoretical part of this paper, we show the following results.

\begin{theorem}\label{thm:existence_even}
 For any $n \in \mathbb{N}$ and $\delta \in \mathbb{R}$, $h_{2n} \neq 0$.
\end{theorem}
This implies that all monic even-degree polynomials $p_{2n,\delta}(x)$ are free of singularities as a function of $\delta$. Furthermore:

\begin{theorem}\label{thm:smalldelta}
For any $n \in \mathbb{N}$ and $\delta \in (-\infty,\delta_0)$, where $\delta_0\approx 2.338$ is 
the smallest root of $\Ai(-\delta)$, $h_n \neq 0$.
\end{theorem}
Here, $\Ai$ is the classical Airy function \cite{NIST:DLMF,Olver:2010:NHMF}. The result shows that all polynomials of any degree exist for all values of $\delta$ on an interval that includes the negative halfline. Of particular interest in practice is that the interior of the interval includes the origin $\delta=0$, a critical value that corresponds to the case where two stationary points coalesce. Hence, as we will see, when two stationary points are sufficiently close, we can find quadrature rules with any desired number of quadrature points.

We denote the roots of $p_{n,\delta}(z)$, if the polynomial exists, by $t_{k,\delta}$, for $k=1,\ldots,n$. Thus, we have $p_{n,\delta}(z) = \prod_{k=1}^n (z-t_{k,\delta})$. As in the real-valued case, the Gaussian quadrature rule corresponds to the exact integral of the interpolating polynomial in the points $t_{k,\delta}$. Using Lagrange interpolation, this leads to a standard expression for Gaussian quadrature weights
\begin{equation}\label{eq:weights_lagrange}
 w_{k,\delta} = \int_\Gamma \frac{p_{n,\delta}(z)}{(z-t_{k,\delta})p_{n,\delta}'(t_{k,\delta})}e^{i \left( \frac{z^3}{3} - \delta z \right)}  {\rm d}z.
\end{equation}
An alternative expression is given further on in \eqref{eq:weights}. The weights are finite if all $t_{k,\delta}$ are distinct, such that the $p_{n,\delta}'(t_{k,\delta})$ does not vanish. This is always the case for real-valued polynomials with strictly positive weight function -- in which case all points are real and distinct and the weights are all positive --  but it is not necessarily so in the complex and oscillatory case.

In order to connect the oscillator of the original integral in \eqref{eq:I} with that of the oscillatory weight in \eqref{eq:orthogonality}, we introduce the scaling $\delta = c \omega^{2/3}$. Finally, we consider the quadrature formula
\begin{equation}\label{eq:rule}
 Q[f]= \frac{1}{\omega^{\frac{1}{3}}} \sum_{k=1}^n w_{k,\delta} f\left( \frac{t_{k,\delta}}{\omega^{1/3}} \right).
\end{equation}
This quadrature rule approximates the contribution of the two stationary points to the original integral \eqref{eq:I}, at least for large $\omega$. This contribution can be singled out by considering the integrand along a path in the complex plane, that is restricted to a neighbourhood of the origin. In order to evaluate the full integral \eqref{eq:I}, one also has to evaluate the contributions of the endpoints. They are computed using the standard numerical steepest descent method as outlined in \S\ref{sect_quad_uniform}, see Fig.~\ref{fig:combined_paths} for a depiction of the line integrals emanating from the endpoints and the contour $\Gamma$.

This leads to the following asymptotic error estimate.

\begin{theorem}\label{thm:error}
Let $f$ be analytic in a disk $D_r$ around the origin of radius $r > 0$. Furthermore, let $\Gamma$ be the concatenation of a straight line from $\infty \times e^{\frac{5i\pi}{6}}$ to $0$, and a straight line from $0$ to $\infty \times e^{\frac{i\pi}{6}}$. Finally, let $\tilde{\Gamma} = \Gamma \cap D_r$ and consider the integral
\begin{equation}\label{eq:Igamma}
 I_{\tilde{\Gamma}}[f] = \int_{\tilde{\Gamma}} f(z)e^{i\omega \left(\frac{z^3}{3}-cz\right)}dz.
\end{equation}

Assume that the polynomial $p_{n,\delta}$ of degree $n$ satisfies \eqref{eq:orthogonality} with $\delta=c \omega^{2/3}$ and that it has $n$ distinct finite roots $t_{n,k}$. For fixed $\delta$ and increasing $\omega$, which implies $c = {\mathcal O}(\omega^{-2/3})$, the error for the quadrature rule $Q[f]$ given by \eqref{eq:rule} is
 \[
  I_{\tilde{\Gamma}}[f] - Q[f] = {\mathcal O}\left(\omega^{-\frac{2n+1}{3}}\right), \qquad \omega \to \infty.
 \]
\end{theorem}

In the statement of the theorem, the integral is explicitly localized around the origin in a disk of radius $r>0$ in order to capture the contribution of the stationary points, as mentioned above. The precise value of $r$ is not relevant for the asymptotic result, as long as $r$ is strictly positive: singularities of the integrand away from the real line have only an exponentially small (in $\omega$) effect on the integral. Note that in the described regime the parameter $c$ decreases asymptotically with increasing $\omega$, hence ultimately the two stationary points at $\pm \sqrt{c}$ are inside the disk with fixed radius $r$.

We describe the numerical method in \S\ref{sect_numerical_method}. Numerical examples are included in \S\ref{sect_experiments} and the  experiments in that section highlight several interesting features of the quadrature approach. These features are analyzed and explained theoretically in \S\ref{sect_analysis}, which includes also proofs of Theorems \ref{thm:existence_even} and \ref{thm:smalldelta} above. A brief asymptotic error analysis is carried out in \S\ref{sect_error_analyis}, leading to the proof of Theorem \ref{thm:error}. The construction of the quadrature rule is detailed in \S\ref{sect_construction}. Here, the major complication is the fact that the rule depends on two parameters, $\omega$ and $c$, and has to be computable efficiently on the fly in applications. The applicability of the method is extended to other integrals with coalescing saddles in \S\ref{s:genericphase}, and an example is shown in \S\ref{sect_application} for a problem that attracts current interest in the literature on numerical methods for oscillatory integrals \cite{dominguez2011filonclenshawcurtis,ledoux2012interpolatory,huybrechs2012superinterpolation}.

\section{The numerical method}\label{sect_numerical_method}

The method of steepest descent is one classical way to derive asymptotic expansions for oscillatory integrals, in which the integration path is explicitly deformed into the complex plane. It goes back to Riemann and Cauchy \cite{bleistein1975asymptotic,wong2001asymptotic}. The purpose of the so-called \emph{numerical} method of steepest descent is to evaluate the resulting line integrals numerically, rather than asymptotically, using Gaussian quadrature rules. Our description is based on \cite{huybrechs2006osc1}, but earlier methods similar in spirit have been described in literature before for specific applications (e.g. \cite{todd54eei,gautschi70ecc,chandlerwilde95ecg}). In this section we develop a generalization of this approach to integrals with coalescing saddle points. Path deformation is only implicit in this case, since the contour is determined by connecting the roots of orthogonal polynomials in the complex plane.

\subsection{The numerical method of steepest descent (NSD)}\label{ss:NSD}

In the method of steepest descent for integral \eqref{eq:I}, the path of integration is deformed onto the steepest descent paths for the oscillator $g(x)$. For integral \eqref{eq:I} this results in:
\begin{itemize}
\item a half infinite path integral $\Gamma_{\{a,b\}}$ through each endpoint of the interval $[a,b]$,
\item and a double infinite path $\Gamma_\xi$ through each stationary point $\xi$ of the oscillator $g(x,c)$.
\end{itemize}
For the details of the numerical scheme, we refer to the references \cite{huybrechs2006osc1,deano2009complexgauss}. The paths are such that the weight function $e^{i \omega g(x)}$ is non-oscillatory along these paths. This is achieved by following a level curve of the real part of $g$. For the case of $g(x,c)$, given by \eqref{eq:osc}, there are two stationary points: $\pm \sqrt{c}$. We denote the associated paths by $\Gamma_+$ and $\Gamma_-$. A typical illustration of the deformed paths is shown in Figure~\ref{fig:paths}.

\begin{figure}[t]
\begin{center}
\begin{overpic}[scale=0.4,unit=1mm]{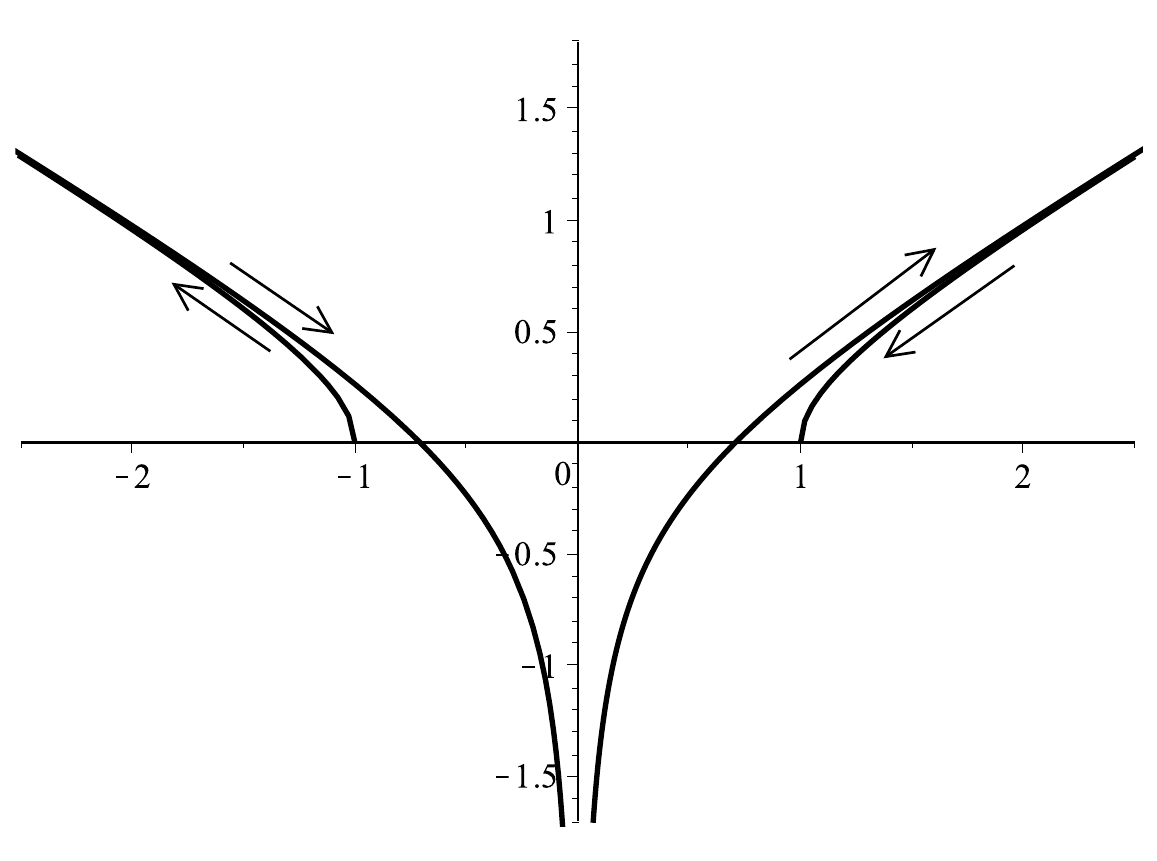}
\put(20,38){$\Gamma_a$}
\put(77,38){$\Gamma_b$}
\put(30,45){$\Gamma_-$}
\put(63,45){$\Gamma_+$}
\end{overpic}
\caption{Steepest descent paths for an oscillatory integral on $[-1,1]$ with two stationary points: two half infinite paths $\Gamma_a$ and $\Gamma_b$ originate from the endpoints $\pm 1$, and two doubly-infinite paths $\Gamma_{+/-}$ pass through the stationary points.}\label{fig:paths}
\end{center}
\end{figure}
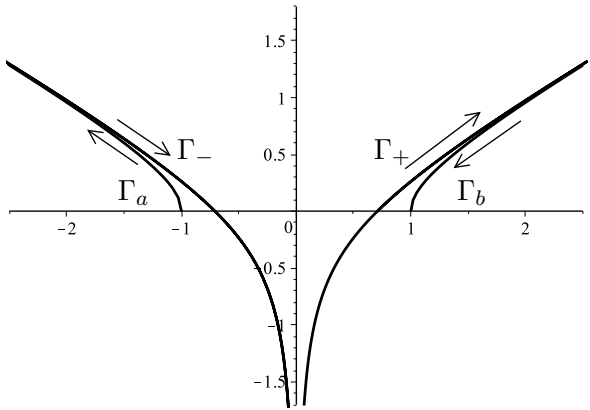

The steepest descent integrals can be parameterized in a way that makes them suitable for Gaussian quadrature. In particular, the half-infinite paths can be written in the form 
\[
 \int_0^\infty u(t) e^{-\omega t} {\rm d}t.
\]
Up to a scaling by $\omega$, this integral can be evaluated with Gauss-Laguerre quadrature~\cite{huybrechs2006osc1}. Interestingly, it can be shown that an $n$-point Gauss-Laguerre rule carries an error of the order ${\mathcal O}\left(\omega^{-2n-1}\right)$ for this integral. Truncating the asymptotic expansion of the same integral after $n$ terms leads to an error of size ${\mathcal O}\left(\omega^{-n-1}\right)$. The difference by a factor of nearly two in the exponents is due to the Gaussian nature of the quadrature and is the reason for the (asymptotic) optimality of this approach. For large values of $\omega$, i.e. for very highly oscillatory integrals, the approximation error is likely to be very small even when $n$ is a small number.

Similarly, the doubly-infinite paths can be written in the form (see \cite{deano2009complexgauss})
\begin{equation}\label{eq:hermite}
 \int_{-\infty}^\infty u(t) e^{-\omega t^2} {\rm d}t,
\end{equation}
well suited for Gauss-Hermite quadrature. The error in this case is ${\mathcal O}\left(\omega^{-\frac{2n+1}{2}}\right)$. Here, too, the exponent is twice as large as in the error term of an $n$-term truncated asymptotic expansion. In both cases, the quadrature points correspond to function evaluations of $f$ at points that lie exactly on one of the steepest descent paths shown in Figure \ref{fig:paths}.

However, the integrand $u(t)$ in the integral corresponding to the stationary point at $-\sqrt{c}$ has a singularity in the complex plane, that arises from the other stationary point at $\sqrt{c}$ (and vice-versa). As the parameter $c$ decreases, this singularity of $u(t)$ moves closer towards the real axis in \eqref{eq:hermite}. As such, though the method in principle applies for any $c > 0$, accuracy deteriorates for small $c$ as the convergence rate of Gauss-Hermite quadrature decreases. We illustrate this with a numerical experiment further on. At $c=0$, the saddle points coincide and $g'(x)$ has a double root. In this case, Gauss-Hermite quadrature no longer applies. Like the asymptotic expansions themselves, the numerical method for $c>0$ does not apply to the case $c=0$.

%

\subsection{Uniform asymptotic expansion of oscillatory integrals}\label{ss:uniform}

Uniform a\-symp\-totic expansions for oscillatory integrals with coalescing saddle points are usually formulated in terms of the Airy function and its derivative \cite{bleistein1975asymptotic,olver1974specialfunctions}. The uniform expansion is not unique and variations are possible, but one statement is:
\begin{equation}\label{eq:airy_expansion}
\int_\Gamma f(x)e^{i\omega\left(\frac{x^3}{3}-cx\right)}dx\sim \frac{1}{\omega^{\frac{1}{3}}}\sum_j\left[\frac{a_{2j}}{\omega^{j}}\Ai(-c\omega^{\frac{2}{3}})\right]+\frac{1}{\omega^{\frac{2}{3}}}\sum_j\left[\frac{a_{2j+1}}{\omega^{j}}\Ai'(-c\omega^{\frac{2}{3}})\right].
\end{equation}
Here, the coefficients $a_j$ are determined by $f,c$ and $\omega$ in a complicated way as follows. First, the function $f$ is written in the form $f(x) = (x^2-c) f_1(x) + f_2(x)$. Next, integration by parts is performed for the first term in this sum (note that $(x^2-c) = g'(x,c)$). This results in an Airy function and a new integral along $\Gamma$, after which the process is repeated recursively.

There are several numerical issues with expansion \eqref{eq:airy_expansion}. First, the asymptotic expansion does not necessarily converge. This is true in general: asymptotic expansions rarely converge \cite{boyd1999devil}. One exception is for polynomials $f$, in which case the expansion terminates and the above expression becomes exact. The second issue is less often reported, but equally profound: the numerical computation of the coefficients $a_j$ is exceedingly difficult and the evaluation of the truncated expansion is numerically unstable \cite{temme1997airy}. This is due essentially to the fact that the coefficients $a_j$ are obtained after repeated application of L'H{\^o}spital's rule, leading in practice for small $c$ to extensive cancellation errors.

In spite of its numerical issues, expansion \eqref{eq:airy_expansion} is valid uniformly for $c \in [-C_1,C_2]$ for any $C_1,C_2 > 0$. For large $c$, the Airy functions can be expanded asymptotically and the Poincar{\'e}-type asymptotic expansion of the integral is recovered. For small $c$, the Airy function precisely captures the behaviour of the coalescing saddle points.

\subsection{The uniform numerical method of steepest descent (UNSD)}\label{sect_quad_uniform}

The numerical method of steepest descent starts with an explicit path deformation, followed by the application of a Gaussian quadrature rule. The points of the quadrature rule lie exactly on the steepest descent paths. Here, we reverse the order of the steps. We formulate a Gaussian quadrature rule, one that evaluates the path integrals $\Gamma_+$ and $\Gamma_-$ simultaneously. The roots of the orthogonal polynomials implicitly correspond to a contour in the complex plane. The resulting path deformation is illustrated in Figure~\ref{fig:combined_paths}.

\begin{figure}[t]
\begin{center}
\begin{overpic}[scale=0.5,unit=1mm]{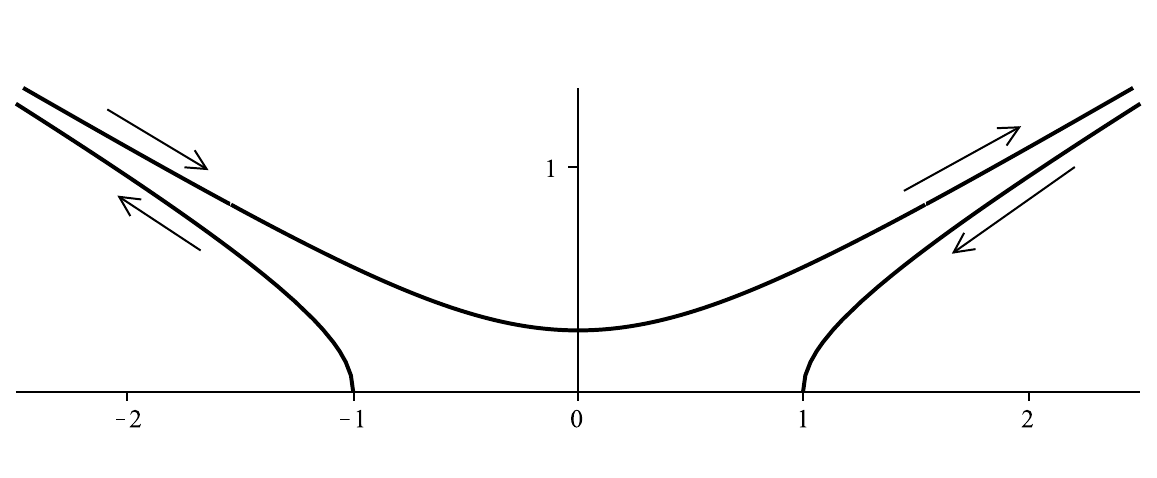}
\put(16,16){$\Gamma_a$}
\put(55,16){$\Gamma$}
\put(94,25){$\Gamma_b$}
\end{overpic}
\caption{The figure shows the paths through the endpoints $\Gamma_a$ and $\Gamma_b$ as in Figure \ref{fig:paths}, but the paths $\Gamma_+$ and $\Gamma_-$ are combined into a single contour $\Gamma$.}\label{fig:combined_paths} 
\end{center}
\end{figure}

Let us be more precise. For the time being, we assume that $f$ is analytic in a sufficiently large region of the complex plane and does not grower faster than exponentially at infinity, such that all path deformations are justified by Cauchy's integral theorem. We assume that the endpoint integrals are treated using Gauss-Laguerre quadrature as before and focus on the doubly-infinite paths from now on. Define $\Gamma$ to be any contour connecting the points $\infty \times e^{\frac{5i\pi}{6}}$ and $\infty \times e^{\frac{i\pi}{6}}$ at infinity, as illustrated in Fig.~\ref{fig:combined_paths}. In particular, we could take $\Gamma$ to be the piecewise linear contour defined in Theorem~\ref{thm:error}, or the union $\Gamma = \Gamma_+ \cup \Gamma_-$ of the steepest descent contours passing through the stationary points (see Fig.~\ref{fig:paths}).

Upon the change of variables
\begin{equation}\label{eq:change_of_variables}
 t=\omega^{\frac{1}{3}}x \qquad \mbox{and}\qquad \delta=c\omega^{\frac{2}{3}},
\end{equation}
we have (recall the definition \eqref{eq:osc} of $g$)
\[
 \int_\Gamma f(x) e^{i \omega g(x,c)} {\rm d}x = \omega^{-1/3} \int_\Gamma f\left( \frac{t}{\omega^{1/3}}\right) e^{i g(t,\delta)} {\rm d}t.
\]
Note that the path of integration does not change under this scaling, assuming $f$ is analytic in a sufficiently large region. The advantage of the latter reformulation is that the oscillator depends on only one parameter, $\delta$, rather than two.

Consider a family of monic orthogonal polynomials $p_{n,\delta}(x)$ satisfying the orthogonality conditions \eqref{eq:orthogonality}. These polynomials should for the time being be considered only formally orthogonal, as their existence is not guaranteed for this oscillatory weight function in the complex plane. Assuming the monic orthogonal polynomial of degree $n$ exists uniquely for some particular value of $\delta$, i.e. $h_{n-1} \neq 0$, there is a corresponding Gaussian quadrature rule with $n$ points and weights. This rule is suitable for weighted integration along $\Gamma$,
\[
 \int_\Gamma u(t) e^{i g(t,\delta)} {\rm d}t \approx \sum_{k=1}^n w_{k,\delta} \, u(t_{k,\delta}).
\]
It is emphasized in this notation that the points and weights depend on the parameter $\delta$. Note that this rule plays the same role as the Laguerre and Hermite rules before.

For the original integral we arrive, after undoing the transformation \eqref{eq:change_of_variables}, at the quadrature rule \eqref{eq:rule}. 
This is the quadrature rule we propose and investigate in this paper. It is denoted in the following by UNSD, for uniform numerical method of steepest descent. Note that any phase function with two stationary points can be mapped to the canonical case $g(x,c)$ by a smooth change of variables.

This quadrature rule can be seen as the numerical equivalent of the uniform asymptotic expansion. It is clear that the rule applies to any value of $c$, including $c=0$. That case corresponds to $\delta=0$ in our notation, and the stationary points coalesce. A special case treatment for such a degenerate stationary point was described before \cite{deano2009complexgauss}, and the current quadrature method simply reduces to that case exactly. As it turns out, for large $c$ the rule above also reduces (numerically) to two individual applications of Gauss-Hermite along the steepest descent paths $\Gamma_+$ and $\Gamma_-$ separately, shown in Fig. \ref{fig:paths}. For large $c$, and thus for large $\delta$, the orthogonal polynomial $p_{2n,\delta}(x)$ of degree $2n$ is close to the product of two Hermite polynomials of degree $n$. Like the uniform asymptotic expansion reduces to the regular Poincar{\'e}-type expansions for large $c$, the uniform numerical scheme reduces to the regular NSD scheme in the same regime.

A disadvantage of the quadrature rule that should be noted is its dependence on the parameter $\delta$. Since this parameter may have any value in applications, the quadrature rule has to be computable on the fly. Thus, unless a priori computations are feasible, the construction of the quadrature rule is an integral part of the cost of the application of the scheme. For phase functions different from \eqref{eq:osc}, the cost of the change of variables has to be taken into account as well.

\section{Numerical experiments}\label{sect_experiments}

Before analyzing the polynomials and the associated quadrature rule, we perform a number of experiments to illustrate the accuracy that can be achieved for varying values of $c$ and $\omega$. We show results for the simple analytic function 
\[
 f(x) = \sin(4x).
\]
This function is simple, yet not entirely innocent as it grows exponentially in the complex plane. This growth, as we shall see, offsets the benefit of steepest descent deformation for small $\omega$. A more interesting example is given in \S\ref{sect_application} later on.

\setlength{\figurewidth}{9cm}
\setlength{\figureheight}{7cm}

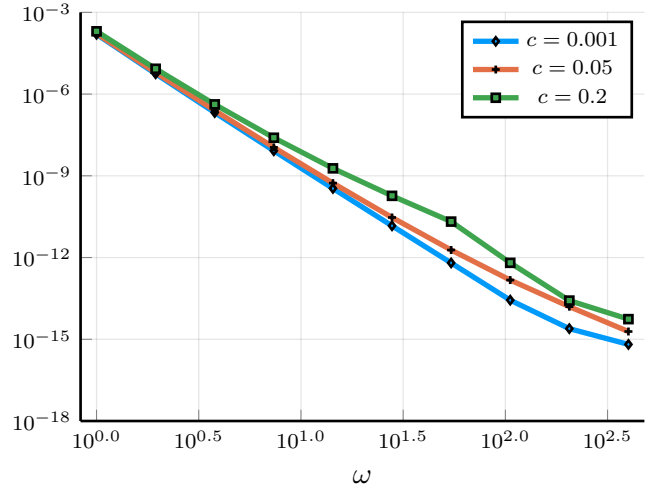
\begin{figure}[t]
\centering
\begin{small}
\input{unsd_abs_error_v2.tikz}
\end{small}
\caption{Approximation error (absolute error) of the UNSD method as a function of $\omega$ for $f(x)=\sin(4x)$ with a $6$-point rule ($n=6$) and for various values of $c$.}\label{fig:error_omega}
\end{figure}

\subsection{Asymptotic order of convergence}

Our first observation is that the error of the UNSD scheme decreases algebraically with increasing $\omega$. This is illustrated in Figure~\ref{fig:error_omega}. The error decreases as a function of $\omega$ at a similar rate for all $c$, though it appears to be smaller for smaller $c$. Though the number of quadrature points is very modest, only $n=6$, absolute errors on the order of $1e-10$ are reached for $\omega$ as small as $100$.

In contrast, for small values of $\omega$ the error is fairly large. As mentioned above, this is exacerbated by the exponential growth of $f$ in the upper half of the complex plane. However, the lack of convergence for small $\omega$ is typical. This cause is most obvious from the scaling of the roots in \eqref{eq:rule}: the term  $\omega^{-1/3} t_{k,\delta}$ implies that the quadrature points $t_{k,\delta}$ are mapped closer to the real line for increasing $\omega$, but away from the real line for decreasing $\omega$. The rule as formulated here does not have a proper limit $\omega \to 0$. Yet, recall that for small $\omega$ the original integral is non-oscillatory, and a straightforward approach is to evaluate it by other means in this regime.

\subsection{Accuracy for small \texorpdfstring{$c$}{c} and comparison to NSD}

In Figure~\ref{fig:unsd_vs_hermite} we compare the error of the proposed UNSD-method with the usual NSD scheme based on Hermite polynomials described in \S\ref{ss:NSD}, and analyzed in \cite{deano2009complexgauss}. For NSD we evaluate the two line integrals through $\pm \sqrt{c}$ using Gauss-Hermite quadrature with $n$ points each. For UNSD we evaluate their sum using a single quadrature rule with $2n$ points in total. It is clear that the NSD method fails for small $c$: the error blows up as $c$ tends to $0$. In contrast, the error of the UNSD scheme is uniformly small in $c$, and that is the motivation of this work.

\setlength{\figurewidth}{9cm}
\setlength{\figureheight}{7cm}

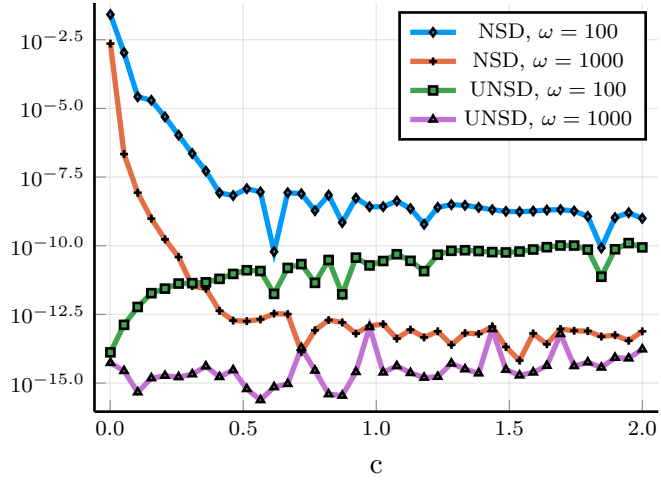
\begin{figure}[t]
\centering
\begin{small}
\input{nsd_vs_unsd_v2.tikz}
\end{small}
\caption{A comparison of the accuracy of two applications of the Gauss-Hermite rule versus the UNSD rule, for $n=6$ and $f(x)=\sin (4x)$, as a function of $c$.}\label{fig:unsd_vs_hermite}
\end{figure}

It is also interesting to observe in this figure that the UNSD and NSD rules appear to exhibit very similar errors for large $c$. It turns out that in this regime the quadrature points are nearly the same. We illustrate the location of the quadrature points in Figure~\ref{fig:roots_split}.
The left panel shows the location in the complex plane of the quadrature points for various $\delta$. When $\delta$ is small, the roots lie on a single curve. For larger values of $\delta$, the roots seemingly cluster in two separate curves. These two clusters are illustrated again in the right panel and compared to the quadrature points of the NSD scheme: they are indeed very close to each other.

\setlength{\figurewidth}{6cm}
\setlength{\figureheight}{5cm}

\begin{figure}[!t]
\centering
\subfigure{
\input{delta_points.tikz}
}
\subfigure{
\input{unsd_hermite_points.tikz}
}

\caption{Illustration of the roots of $p_{n,\delta}$. The left panel shows that the roots seem to lie on a single curve for small $\delta$, but split in two groups for larger $\delta$ ($n=14$, $\delta$ increases from top to bottom with values $\delta=0,2,4,\ldots,12$). The right panel shows that the roots of the UNSD-rule and the NSD rule for a large value of $\delta$ are close to each other ($\delta=10$).}\label{fig:roots_split}
\end{figure}
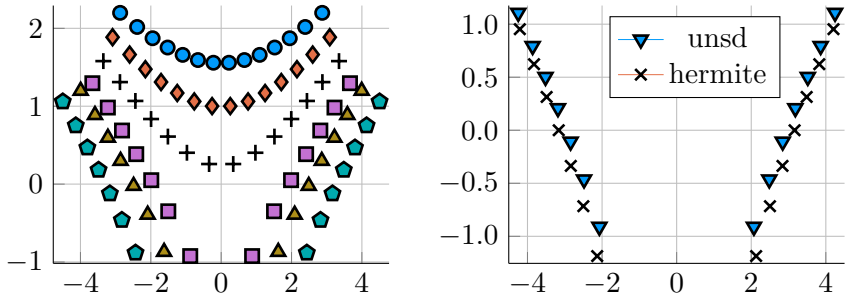

\subsection{Issues with quadrature rules with an odd number of points}

The previous experiments have shown that the proposed UNSD method works well for all values of $c$. However, thus far we have used only quadrature rules with an even number of points. This is based on Theorem \ref{thm:error}, which only guarantees existence of the orthogonal polynomials for even $n$. One quickly observes that there are serious issues with the quadrature rules with an odd number of points. This is illustrated in Figure \ref{fig:odd_vs_even}: the quadrature rules with odd values of $n$ lead to large errors, at least for some isolated values of $c$. In agreement with Theorem \ref{thm:smalldelta} there are no problems of existence for small values of $c$, which is the intended regime for the quadrature rule. Yet, it seems advisable to use only an even number of points in computations in general.

\setlength{\figurewidth}{6cm}
\setlength{\figureheight}{5cm}

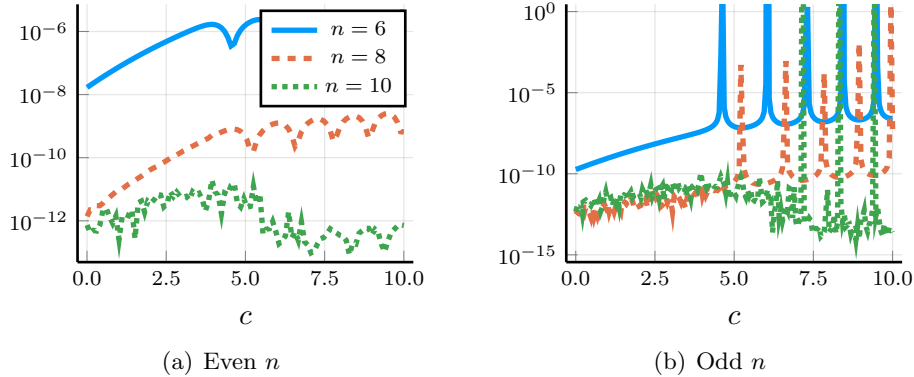
\begin{figure}[t]
\begin{center}
\subfigure[Even $n$]{
\begin{small}
\input{error_even_n_v2.tikz}
\end{small}
}
\subfigure[Odd $n$]{
\begin{small}
\input{error_odd_n_v2.tikz}
\end{small}
}
\caption{Illustration of the accuracy of the UNSD rule applied to $f(x)=\sin x + \cos x$, for various numbers of points $n$ and as a function of $c$ (using $\omega=1$, hence by \eqref{eq:change_of_variables} $\delta=c$). For odd values of $n$, the quadrature error exhibits large spikes at particular values of $c$. These spikes correspond to non-existence of the underlying orthogonal polynomials. No spikes occur for small values of $c$.}\label{fig:odd_vs_even}
\end{center}

\end{figure}

\section{Proofs of Theorems \ref{thm:existence_even} and \ref{thm:smalldelta} }
\label{sect_analysis}

In this section we establish analytical results that completely describe the features we have observed with the experiments. The proofs rely on the relation of the moments \eqref{eq:moments} with the Airy function and its derivative.

In order to simplify the expressions, as well as to make connections to existing literature on orthogonal polynomials, in this section we adopt the integration contour $C$ of the standard Airy function (defined below), rather than contour $\Gamma$ in \eqref{eq:orthogonality}. We note that $C$ is simply a rotation of $\Gamma$, and the polynomials $P_n(s)$ in this section are related to $p_{n,\delta}$ by
\[
 P_n(s) = i^n p_n^{-x}(-i s).
\]

\subsection{Airy function and determinants}
The Airy function $\Ai$ is the solution of the Airy differential equation $y''(x) = xy(x)$
with the asymptotic behavior
\begin{equation} \label{eq:Airyasymp}
	\Ai(x) = \frac{1}{2\sqrt{\pi}x^{1/4}} e^{-\frac{2}{3} x^{3/2}} \left(1 + O(x^{-3/2})\right)
\end{equation}
as $x \to +\infty$. It has an integral representation
\begin{equation} \label{eq:Airyintegral} 
	\Ai(x) = \frac{1}{2\pi i} \int_C e^{- \frac{1}{3} s^3 + xs} ds 
	\end{equation}
where $C$ is an infinite contour in the complex $s$-plane from $ \infty e^{-2\pi i/3}$
to $\infty e^{2\pi i/3}$. Then 
\[ \Ai^{(k)}(x) = \frac{1}{2\pi i} \int_C s^k e^{- \frac{1}{3} s^3 + xs} ds \]
and a comparison with the definition \eqref{eq:moments} of $\mu_k$ shows that
\begin{equation} \label{eq:Airymuk} 
	\mu_k =  2 \pi (-i)^k \Ai^{(k)}(-\delta).
	\end{equation}

We define $D_0(x) = 1$ and for $n \geq 1$,
\begin{equation} \label{eq:Airydet} 
	D_n(x) = \det \left[ \frac{d^{j+k-2}}{dx^{j+k-2}} \Ai(x) \right]_{j,k=1, \ldots, n}.
	\end{equation}
Then $D_n$ is an entire function, which is related to the determinant $h_n = \det H_n$
as follows.
\begin{lemma}
We have
\begin{equation} \label{eq:Airydethn} 
	h_n = \det H_n =  (2\pi)^n (-1)^{n(n-1)/2}  D_n(-\delta) 
	\end{equation}
\end{lemma}
\begin{proof}
Because of \eqref{eq:hankel} and \eqref{eq:Airymuk} we have
\[ h_n = (2 \pi)^n \det \left[ (-i)^{j+k-2} \Ai^{(j+k-2)}(-\delta) \right]_{j,k=1, \ldots, n}. \]
We take out a factor $(-i)^{j-1}$ out of row $j$ and a factor $(-i)^{k-1}$ out of column $k$
of the determinant, for each $j,k = 1, \ldots, n$.
This results in the factor
\[ \prod_{j=1}^n (-i)^{j-1} \cdot \prod_{k=1}^n (-i)^{k-1} = (-1)^{n(n-1)/2}. \]
and the formula \eqref{eq:Airydethn} follows.
\end{proof}

In view of \eqref{eq:Airydethn}, in order to prove Theorems \ref{thm:existence_even} and \ref{thm:smalldelta} we need to prove that $D_n$ has no real zeros if $n$ is even,
and has no zeros in $(\iota_1, \infty)$, where $\iota_1$ is the largest zero of the Airy function
if $n$ is odd.
Recall that $\Ai$ has an infinite number of zeros, all  negative and simple, that are
usually denoted as $0 > \iota_1 > \iota_2 > \iota_3 > \cdots$.

This claim is easy to verify for small values of $n$. Clearly $D_1 = \Ai$
has no zeros in $(\iota_1, \infty)$. For $n=2$ we have by \eqref{eq:Airydet}
\[ D_2(x) = \Ai(x) \Ai''(x) - \Ai'(x)^2 = x \Ai(x)^2 - \Ai'(x)^2, \]
where we used the Airy differential equation $\Ai''(x) = x \Ai(x)$. Then by 
a simple calculation
\begin{equation} \label{eq:D2formula} 
	D_2'(x) = \Ai(x)^2, 
	\end{equation}
and, since $D_2(x) \to 0$ as $x \to +\infty$, we get
\[  D_2(x) = - \int_x^{\infty} \Ai(s)^2 ds < 0, \qquad \text{for } x \in \mathbb R. \]
Thus $D_2$ has no real zeros indeed.

The proof for general $n$ will follow from certain identities for the Airy 
determinants.

\begin{remark}
\normalfont
The Airy determinant \eqref{eq:Airydet} appears in the work of Forrester and Witte \cite{FW} within the context of random matrix theory, see formula (1.25). It arises in the soft edge scaling limit of the expectation value of the $n$-th power of the characteristic polynomial of a GUE matrix. From this interpretation it is natural that $D_n$ does not vanish if $n$ is even. The quantity $D_2$ is (up to a sign) the density of the Airy point process.
\end{remark}

\begin{remark}
\normalfont
The Airy determinant \eqref{eq:Airydet} is also a building block for the construction of Airy like solutions of the Painlev{\'e} II equation, see for example Theorem 4 in \cite{Clarkson} and references therein.
\end{remark}

\subsection{Identities for Airy determinants}

\begin{lemma} \label{lem:Dnidentity}
We have the differential identity
\begin{equation} \label{eq:Airydetdiffid} 
	D_{n-1} D_{n+1} = D_n D_n'' - (D_n')^2. 
	\end{equation}
\end{lemma}
\begin{proof}
This can be proved as in [9, Lemma 6] using the connection with
orthogonal polynomials, and their recurrence coefficients (see also
the proof of Proposition \ref{prop:Airydet}).

However, it can also be proved from Jacobi's identity for determinants,
see \cite[formula (0.8.4.1)]{horn1985}. A special case of this formula is the following. For a square matrix
$A$ we use $A[i;j]$ to denote the matrix obtained from $A$ by deleting
row $i$ and column $j$, and we use $A[ij;kl]$ for the matrix obtained
from $A$ by deleting rows $i$ and $j$ and columns $k$ and $l$,
where $i<j$ and $k< l$. Then
\begin{equation} \label{eq:JacobiId} 
	\det A  \cdot \det A[ij;kl] = \det A[i;k] \cdot \det A[j;l]
	- \det A[i;l] \cdot \det A[j;k]. 
	\end{equation}
Now let $A$ be the $(n+1) \times (n+1)$ matrix
\[ A = \left[ \frac{d^{j+k-2}}{dx^{j+k-2}} \Ai(x) \right]_{j,k=1, \ldots, n+1}. \]
Then $\det A = D_{n+1}$, $\det A[n+1; n+1] = D_n$, 
$\det A[n n+1; n n+1] = D_{n-1}$, $\det A[n; n+1] = \det A[n+1; n] = D_n'$
and $\det A[n;n] = D_n''$ and using this in \eqref{eq:JacobiId}
we obtain \eqref{eq:Airydetdiffid}.
\end{proof}

\begin{lemma} \label{extralemma}
$x \mapsto D_n(x)$ is  not identically zero for $x \in \mathbb C$.
\end{lemma}
\begin{proof} We use the fact that the Wronskian of a finite number of 
analytic functions vanishes identically, if and only if
the functions are linearly dependent. This is 
due to B\^ocher \cite{Bocher} in 1900, but see \cite{Bostan} for a recent proof.

By \eqref{eq:Airydet} $D_n$ is the Wronskian of the functions $\Ai$, $\Ai'$, \ldots, $\Ai^{(n-1)}$. If $D_n$ would be identically zero, then the functions would be linearly dependent, and so there would be $c_1, \ldots, c_{n}$, not all $0$, such that $\sum\limits_{j=1}^{n} c_j \Ai^{(n-j)} \equiv 0$. Then $\Ai$ would be a solution of a homogeneous linear differential equation with constant coefficients. This is impossible, since the only functions that are solutions of such ODEs are  polynomials, exponential functions $x \mapsto e^{\lambda x}$, and finite combinations (sums and products) thereof.
\end{proof}

The differential identity \eqref{eq:Airydetdiffid} is well-known in integrable
systems where it is related to the Toda lattice equations, see e.g. \cite{hirota1987}. However the following identity seems to be new. 
It is specific for the Airy determinants \eqref{eq:Airydet} and it 
reduces to the identity \eqref{eq:D2formula} for $n=1$.

\begin{proposition}\label{prop:Airydet}
For each $n \geq 1$ we have
\begin{equation}  \label{eq:Airydetid2}
	\left( \frac{D_{n+1}}{D_{n-1}} \right)'
		= n \left( \frac{D_n}{D_{n-1}} \right)^2. 
\end{equation}
\end{proposition}
\begin{proof}

Let $P_n$ be the monic orthogonal polynomial of degree $n$ for the complex weight
$e^{-\frac{1}{3}s^3 + xs}$ (with $x$ as a complex parameter) 
on the contour $C$. That is
\begin{equation} \label{eq:Pnortho} 
	\int_C P_n(s) s^k e^{-\frac{1}{3} s^3 + xs} ds = 0, \qquad k=0, \ldots, n-1.
\end{equation}
The polynomial $P_n$ exists if and only if $D_n(x) \neq 0$.

By Lemma \ref{extralemma} $D_n$ does not vanish identically,
and since it is an entire function, its zero set
\[ Z_n = \{ x \in \mathbb C \mid D_n(x) = 0 \} \]
is an at most countable set without accumulation point in
the complex plane. Same conclusion holds for $Z_{n-1}$
and $Z_{n+1}$. Take a non-empty open set $\Omega \subset 
\mathbb C \setminus (Z_{n-1} \cup Z_n \cup Z_{n+1})$.
Then for $x \in \Omega$ the polynomials $P_{n-1}$, $P_n$
and $P_{n+1}$ exist and they are related by a three term recurrence
\begin{equation}
\label{eq:recurrence_arno}
sP_n(s) = P_{n+1}(s) + b_n P_n(s) + a_n P_{n-1}(s)
\end{equation}
where
\begin{equation} \label{eq:anbn} 
	\begin{aligned}
	a_n & = \frac{D_{n-1} D_{n+1}}{D_n^2},   \\
	b_n & = \frac{D_{n+1}'}{D_{n+1}} - \frac{D_n'}{D_n}.
	\end{aligned}
	\end{equation}
As a function of $x \in \Omega $ they satisfy the Toda differential equations
\begin{equation} \label{eq:anbnToda}
	\begin{aligned}
	\frac{da_n}{dx} & = a_n(b_n-b_{n-1}), \\
	\frac{db_n}{dx} & = a_{n+1}-a_n.
	\end{aligned}
	\end{equation}	
The equations \eqref{eq:anbn} and \eqref{eq:anbnToda} are well-known
and  can be found in various forms in the literature, see e.g.
equations (2.9)  and (2.10) in \cite{sogo1993}, where  $a_n^2$ is used instead of $a_n$ and $\tau_n$
is used for  $D_{n+1}$. The equations \eqref{eq:anbnToda} hold
for recurrence coefficients whenever $x$-dependence of
the orthogonality weight comes in the form $w(s) e^{xs}$.

For the cubic weight $w(s) = e^{-\frac{1}{3} s^3}$ as in \eqref{eq:Pnortho}
the recurrence coefficients satisfy nonlinear difference equations
(sometimes called string equations)
\begin{equation} \label{eq:string}
	\begin{aligned}
	a_n + a_{n+1} & = x - b_n^2, \\
	a_n(b_n+b_{n-1}) & = n.
	\end{aligned}
	\end{equation}
These equations are also given in the recent paper \cite{clarkson2016}, where they were identified as an alternative discrete Painlev\'e I equation. 
Note that the recurrence coefficients in our paper differ by a sign from the ones used by Clarkson et al.
See also the appendix of \cite{vanassche2015}. These are recent papers, but the idea that recurrence coefficients for semi-classical orthogonal polynomials satisfy nonlinear equations 
of Painlev\'e type dates back to at least \cite{magnus1995}.

Inserting \eqref{eq:anbn} into the second string equation \eqref{eq:string}, we find
\begin{equation}
 n = \frac{D_{n-1} D_{n+1}}{D_n^2}  \left( \frac{D_{n+1}'}{D_{n+1}} - \frac{D_n'}{D_n} \right).
\end{equation}
which can be rearranged to give \eqref{eq:Airydetid2} for $x \in \Omega$.
We finally note that both sides of \eqref{eq:Airydetid2}
are analytic in $\mathbb C \setminus Z_{n-1}$, with poles in $Z_{n-1}$. 
Then the identity extends from $x \in \Omega$ to general $x \in \mathbb C$,
by the identity theorem for meromorphic functions.
\end{proof}

\subsection{Behavior for \texorpdfstring{$x \to \infty$}{x to infinity}}

As the last preparation for the proofs of Theorems \ref{thm:existence_even} and \ref{thm:smalldelta}, we have to investigate the behavior of $D_n(x)$ as $x \to +\infty$.

The Airy function $\Ai$ and all its derivatives tend to zero as the argument tends
to infinity along the positive real line. Thus it is clear from \eqref{eq:Airydet} 
that $D_n(x) \to 0$ as $x \to +\infty$. We have the more precise behavior.

\begin{lemma}
We have
\begin{equation} \label{eq:Airydetasymp} 
	D_n(x) =  c_n  x^{-n^2/4}  e^{-\frac{2}{3}n x^{3/2}} \left(1+ O(x^{-3/2}) \right)
\end{equation}
as $x \to +\infty$, with the constant
\[
 c_n = (-1)^{n(n-1)/2}  \frac{\prod_{k=0}^{n-1} k!}{2^{n(n+1)/2} \pi^{n/2}}.
\]
\end{lemma}
\begin{proof}
The lemma is true for $n=0$ and $n=1$, see \eqref{eq:Airyasymp}. 

For general $n$, we use $D_{n+1} = \frac{D_n^2}{D_{n-1}}  a_n$,  see \eqref{eq:anbn}, and the fact that
\[
a_n(x) =  - \frac{n}{2} x^{-1/2} \left(1 + O(x^{-3/2})\right)   \quad \text{ as } x \to +\infty
\]
see \cite[formula (5.4.a)]{clarkson2016}.
Then \eqref{eq:Airydetasymp} follows by an induction argument.
\end{proof}

\subsection{Proofs of Theorems \ref{thm:existence_even} and \ref{thm:smalldelta}}

\begin{proof}
From \eqref{eq:Airydetasymp} it follows that
$ \frac{D_{n+2}(x)}{D_{n}(x)}  \to 0$ as $ x \to +\infty$.
Thus, after changing $n \to n+1$ in \eqref{eq:Airydetid2} and integrating, we obtain
\begin{equation} \label{eq:Dn2integral} 
	D_{n+2}(x) = - (n+1) D_n(x) \int_x^{+\infty} \left( \frac{D_{n+1}(s)}{D_n(s)} \right)^2 ds 
	\end{equation}
provided that $D_n$ has no zeros on the interval $[x,\infty)$. 

Now we can easily prove Theorems \ref{thm:existence_even} and \ref{thm:smalldelta} by an induction argument.
It is true that $D_0 =1$ has no zeros on the real line. 
Suppose $n \geq 0$ is even, and assume that $D_n$ has no zeros on the real line.
Then the integral \eqref{eq:Dn2integral} is valid for every $x \in \mathbb R$.
Since $D_{n+1}$ clearly does not vanish identically (this follows from \eqref{eq:Airydetasymp}),
we immediately find that $D_{n+2}$ has no zeros on the real line as well.

For $n=1$, we have $D_1 = \Ai$, and the Airy function has no zeros on $(\iota_1, \infty)$, where $\iota_1$ is the largest zero of $\Ai$.
Suppose $n \geq 1$ is odd, and assume that $D_n$ has no zeros on $(\iota_1, \infty)$.
Then again it follows from \eqref{eq:Dn2integral} that $D_{n+2}(x)$ has no zeros
 for $x > \iota_1$ as well. 

This proves Theorems \ref{thm:existence_even} and \ref{thm:smalldelta}, where
we recall \eqref{eq:Airydethn}.
\end{proof}

The proof in this paper is based on the relation between the Hankel determinant and Airy determinants. An alternative proof, including more detailed analysis of the recurrence coefficients based on the string equations, can be found in \cite{lejon2016phd}.

\section{Asymptotic error analysis of the quadrature method}
\label{sect_error_analyis}

After showing the existence of the orthogonal polynomials, hence of the proposed quadrature rules, we set out to analyze its convergence characteristics. Since, like the NSD scheme itself, the current quadrature scheme is asymptotic, we focus on the behaviour for $\omega \gg 1$. We consider nearly coalescing saddle points, i.e. $0 < c \ll 1$, since that is the intended regime of the quadrature rule. We want to show that the error decays with $\omega$ at an algebraic rate, as observed in the experiments.

To that end, we investigate the $\omega$-dependence of the quadrature error $I_\Gamma[f] - Q[f]$. Recall the definitions \eqref{eq:Igamma} and \eqref{eq:rule}.

\begin{proof}[Proof of Theorem \ref{thm:error}]
In order to avoid stringent analyticity requirements on $f$, the contour $\Gamma$ is in Theorem \ref{thm:error} truncated to a neighbourhood of the origin in which $f$ is analytic. This does not alter the asymptotic behaviour of the integral, as can be confirmed using integration by parts and noting that the endpoints of the finite contour $\tilde{\Gamma} = \Gamma \cup D_r$ are in the sectors of the complex plane where the oscillator $e^{i \omega g(z,c)}$ becomes exponentially small. In the large $\omega$ limit with fixed $\delta$, the quadrature points $\frac{t_{k,\delta}}{\omega^{1/3}}$ in \eqref{eq:rule} are also in $D_r$. Since the quadrature rule is assumed to be well-defined for given $n$ and $\delta$, it can be applied for any $\omega$ and we investigate the large $\omega$ limiting behaviour.

For simplicity, in what follows we assume that $f$ is extended to a smooth $C^\infty$ function $\tilde{f}$ along $\Gamma$ in such a way that the extended integral
\[
 I_\Gamma[\tilde{f}] = \int_\Gamma \tilde{f}(z) e^{i \omega g(z,c)} dz
\]
exists and is finite. Regardless of the particular extension, we have that $I_{\tilde{\Gamma}}[f] - I_\Gamma[\tilde{f}] = {\mathcal O}(\omega^{-p})$ for all $p > 0$. It follows that $I_{\tilde{\Gamma}}[f] - Q[f] = I_\Gamma[\tilde{f}] - Q[f] + {\mathcal O}(\omega^{-p})$, for all $p > 0$.

Using the change of variables \eqref{eq:change_of_variables} and formula \eqref{eq:rule} the approximation error becomes, up to the above asymptotically small error,
\begin{align*}
 I_\Gamma[\tilde{f}]-Q[f] = \frac{1}{\omega^{\frac{1}{3}}} \int_\Gamma \tilde{f}\left(\frac{t}{\omega^{\frac{1}{3}}}\right)e^{i\left(\frac{t^3}{3}-\delta t\right)} {\rm d}t- \frac{1}{\omega^{\frac{1}{3}}} \sum_{k=1}^n  f\left(\frac{t_{k,\delta}}{\omega^{\frac{1}{3}}}\right)w_{k,\delta}.
\end{align*}
Recall that the quadrature points $t_{k,\delta}$ are the roots of the orthogonal polynomials $p_{n,\delta}$ and the coefficients $w_{k,\delta}$ are the associated weights.

Since the quadrature points are zeros of orthogonal polynomials, polynomials up to degree $2n-1$ are integrated exactly. Because $f$ is analytic in the disk $D_r$, it can be expanded in a Taylor series,
\[
f(x)=\sum_{j=0}^\infty a_j x^j, \qquad |x| < r.
\]
Since the first $2n-1$ terms are integrated exactly, we have in an asymptotic sense that
\begin{small}\begin{align*}
I_\Gamma[\tilde{f}]-Q[f]& \sim \frac{1}{\omega^{\frac{1}{3}}}\sum_{j=0}^\infty a_j\int_\Gamma \frac{t^j}{\omega^{\frac{j}{3}}} e^{i\omega g(t,\delta)} {\rm d}t - \frac{1}{\omega^{\frac{1}{3}}} \sum_{j=0}^\infty a_j \sum_{k=1}^n \frac{t_{k,\delta}^j}{\omega^{\frac{j}{3}}}w_{k,\delta} \\
&=\frac{1}{\omega^{\frac{1}{3}}} \sum_{j=2n}^\infty \frac{a_j}{\omega^{\frac{j}{3}}}\left[\int_\Gamma t^j e^{ig(t,\delta)} {\rm d}t - \sum_{k=1}^n t_{k,\delta}^j w_{k,\delta} \right]\\
&=\frac{1}{\omega^{\frac{1}{3}}} \sum_{j=2n}^\infty \frac{a_j}{\omega^{\frac{j}{3}}}\left[\mu_j(\delta)-\sum_{k=1}^n t_{k,\delta}^j w_{k,\delta} \right].
\end{align*}\end{small}
Note that interchanging summation and integration is not justified here, as the Taylor series of $f$ does not converge outside the circle with radius $R$, yet the integration contour is infinite. The result only holds as an asymptotic series, in exactly the same way as in the well-known Watson's Lemma \cite{bleistein1975asymptotic}.

For fixed $\delta$, the moments $\mu_j(\delta)$ and the sum $\sum_{k=1}^n t_{k,\delta}^j w_{k,\delta}$ are independent of $\omega$. Thus, again for \emph{fixed $\delta$} and for increasing $\omega$, we conclude from the size of the first term in the summation that
\[
I_{\tilde{\Gamma}}[f]-Q[f] = \mathcal{O}\left(\omega^{-\frac{2n+1}{3}}\right).
\]
This concludes the proof.
\end{proof}
Note that we have taken a rather simplified approach here. There are two parameters in the original problem, $c$ and $\omega$, and we have chosen to fix $\delta=c \omega^{2/3}$. The case of fixed $\delta$ and increasing $\omega$ implies that $c$ shrinks like $\omega^{-2/3}$. The rate of decay as a function of $\omega$ we arrive at in this regime is exactly the same as in the case where $\delta=0$ \cite{deano2009complexgauss}.

An analysis for the case of fixed $c$ and increasing $\omega$ would be considerably more involved, as this regime implies that $\delta$ grows like $\omega^{2/3}$. Hence, this regime includes the transition from a single cluster of quadrature points to a double cluster. Still, we know the outcome in the regime of larger $c$ as well, from the existing analysis of the NSD scheme. In NSD we can use $n$ Gauss-Hermite points for a single stationary point integral and achieve $\mathcal{O}\left(\omega^{-\frac{2n+1}{2}}\right)$ error. Using $n/2$ points for the first stationary point and $n/2$ points for the second, we achieve $\mathcal{O}\left(\omega^{-\frac{n+1}{2}}\right)$ error in total. Thus, the UNSD scheme transitions from $\mathcal{O}\left(\omega^{-\frac{2n+1}{3}}\right)$ behaviour for small $c$ to 
$\mathcal{O}\left(\omega^{-\frac{n+1}{2}}\right)$ behaviour for large $c$. The error for large $c$ is larger, at least asymptotically, and this effect was clearly visible in Fig.~\ref{fig:unsd_vs_hermite}.




\section{Construction of the quadrature rule}
\label{sect_construction}

The Gaussian quadrature rule in this paper is unconventional because it has points in the complex plane. Unfortunately, this means that standard methods for the computation of Gaussian quadrature rules have to be amended. Two issues we address here are the stable numerical computation of recurrence coefficients, and the accurate computation of quadrature weights from the Jacobi matrix eigenvalue problem. We start with the latter, in some detail in order to properly explain the modification.

\subsection{The complex-symmetric Jacobi matrix}

Usually, for real-valued problems, a Gaussian quadrature rule is found from an eigenvalue problem involving the symmetric Jacobi matrix. This is a tridiagonal matrix that is defined in terms of the three-term recurrence coefficients of the orthogonal polynomials. In this context, the conventional form in which the three-term recurrence relation is expressed (for the \emph{monic} polynomials $p_k$) is
\begin{equation}
\label{eq:recurrence_monic}
 p_{k+1}(x) = (x - \alpha_k) p_k(x) - \beta_k p_{k-1}(x).
\end{equation}
Here, the recurrence coefficients are given by the standard formulas
\begin{equation}\label{eq:alpha_beta}
 \alpha_k = \frac{ (x p_k, p_k) }{(p_k,p_k) } \qquad \mbox{and} \qquad \beta_k = \frac{(p_k,p_k)}{(p_{k-1},p_{k-1})},
\end{equation}
where the bilinear form $(f,g)$ is given by
\begin{equation}\label{eq:bilinear_form}
 (f,g) = \int_\Gamma f(z) g(z) e^{i \left(\frac{z^3}{3}-\delta z \right) } {\rm d}z.
\end{equation}
The complex-valued orthogonal polynomials satisfy a three-term recurrence relation because the orthogonality condition \eqref{eq:orthogonality} is non-hermitian, i.e., we have $(x p, q) = (p, x q)$, rather than $(x p, q) = \overline{(p, x q)}$. Hence, the standard formulas remain valid in the complex case.

Note that the bilinear form \eqref{eq:bilinear_form} is not an inner product, due to the oscillatory weight function, hence the problem of existence of $p_n$.

The corresponding recurrence relation for the normalized polynomials $\pi_k(x)$, normalized in the sense that $(\pi_k,\pi_k) = 1$, is \cite{gautschi2004opq},
\begin{equation}
\label{eq:recurrence_normalized}
 \sqrt{\beta_{k+1}} \pi_{k+1}(x) = (x - \alpha_k) \pi_k(x) - \sqrt{\beta_k} p_{k-1}(x),
\end{equation}
where we set $\beta_0 = \mu_0$. This leads to the complex-symmetric Jacobi matrix,
\begin{align*}
J_n=\begin{pmatrix}
\alpha_0 &\sqrt{\beta_1} &0&0&\ldots &0\\
\sqrt{\beta_1} & \alpha_1 &\sqrt{\beta_2}  &0&\ldots &0\\
0& \sqrt{\beta_2} &\alpha_2 & \sqrt{\beta_3}  &\ldots &0\\
\vdots &\vdots &\vdots & \ddots &\ddots &\vdots\\
0& 0&0&\ldots &\alpha_{n-1}&\sqrt{\beta_n} \\
0& 0&0&\ldots &\sqrt{\beta_n} &\alpha_n
\end{pmatrix}.
\end{align*}
The roots of the polynomial are the eigenvalues of this matrix. The matrix is complex-symmetric, i.e. it is symmetric ($J_n = J_n^T$) but not hermitian ($J_n = J_n^*$). One can show from the string equations (see \eqref{eq:string_beta}-\eqref{eq:string_alpha} further below) that $\alpha_k$ is purely imaginary, and $\beta_k$ is real and positive, hence the square root is well defined \cite{lejon2016phd}.

\subsection{Eigenvalue decomposition and quadrature weights}

The eigenvector $\mathbf{v}_j$ corresponding to eigenvalue $x_j$ is a normalized vector, where each entry is proportional to $\pi_i(x_j)$:
\begin{equation}\label{eq:eigenvector}
 \mathbf{v}_j = c \left[\begin{array}{c}
                   \pi_0(x_j) \\
                   \pi_1(x_j) \\
                   \ldots \\
                   \pi_{n-1}(x_j)
                  \end{array}\right].
\end{equation}
The constant $c$ is such that $\Vert \mathbf{v}_j \Vert = 1$. In the complex-valued case, this means that
\begin{equation}\label{eq:c}
 c = \left( \frac{1}{\sum_{k=0}^{n-1} \pi_k(x_j) \overline{\pi_k(x_j)}} \right)^{1/2}.
\end{equation}
An explicit formula for the corresponding quadrature weight $w_j$ is \cite{gautschi2004opq}
\begin{equation}\label{eq:weights}
 w_j = \frac{1}{\sum_{k=0}^{n-1} \pi_k(x_j)^2},
\end{equation}
involving as above the orthonormal polynomials given by \eqref{eq:recurrence_normalized}.\footnote{Note that one can not in general rule out division by zero in this expression for the weights: the sum of squares in the denominator is not necessarily strictly positive in the complex case, as it is in the real case. An alternative and generally applicable expression for the weights is the weighted integral of the Lagrangian basis polynomials $l_{n,j}(t)$ for polynomial interpolation: $w_j = \int_\Gamma l_{n,j}(t) e^{i g(t,\delta)} {\rm d}t$. Hence, the weights are well defined as long as the Lagrangian polynomials are well defined, and this is the case if all roots $t_{k,\delta}$ are distinct. An issue may arise if the orthogonal polynomial $p_{n,\delta}$ has roots of higher multiplicity. In that case, however, a Gaussian-like quadrature rule would still exist based on Hermite interpolation using derivatives, rather than Lagrangian interpolation. Strictly speaking, we have not excluded that case for the polynomials of this paper. Yet, coalescence of roots was never observed in experiments.}

In the real-valued case, we find the simple relation $w_j = c^2$. Moreover, from $\pi_0(x) = \sqrt{\mu_0}$, we can deduce $c$ and, hence, compute the quadrature weight from the first entry of the eigenvector $\mathbf{v}_j$:
\begin{equation}\label{eq:first_entry}
 w_j = c^2 = \mu_0 v_{j1}^2.
\end{equation}
This is a standard step in the well-known Golub--Welsch algorithm, which has $O(n^2)$ computational complexity overall by exploiting the tridiagonal structure of the Jacobi matrix \cite{golub1969gauss}. The final step \eqref{eq:first_entry} needs modification in the complex case, because the expressions for $w_j$ and $c^2$ are no longer the same due to the appearance of complex conjugates in \eqref{eq:c}.

In the complex case, we can still determine $c = \sqrt{\mu_0} v_{j1}$ from the first entry of the eigenvector $\mathbf{v}_j$. Hence, from \eqref{eq:eigenvector}, we can find the values $\pi_k(x_j)$ and we can explicitly evaluate the sum in \eqref{eq:weights} to find the weight $w_j$. This does require the full eigenvalue decomposition of the Jacobi matrix and leads to $O(n^3)$ computational complexity.

Alternatively, once the points $x_j$ have been determined (in $O(n^2)$ operations as with the Golub--Welsch algorithm), one can evaluate the orthonormal polynomials $\pi_k(x_j)$ using their recurrence relation \eqref{eq:recurrence_normalized}. This requires only $O(n^2)$ operations, and so does the explicit summation of \eqref{eq:weights} for the $n$ quadrature weights. With this modification, the computational cost of the quadrature rule remains $O(n^2)$.

\subsection{Computation of the recurrence coefficients for varying \texorpdfstring{$\delta$}{delta}}

It remains to compute the recurrence coefficients, which is a complication because they depend on the parameter $\delta$. A popular strategy is the Stieltjes algorithm \cite{gautschi2004opq}, which is based on a discrete approximation to the bilinear form \eqref{eq:bilinear_form}. For example, we may use the trapezoidal rule on a parabolic contour $P$ that connects the appropriate points at infinity. However, we found this to be numerically stable only up to moderate values of $n$. The main complication is the growth of the polynomials $p_{n,\delta}$ in the complex plane, which is similar to the growth of real-valued orthogonal polynomials when evaluated away from the real line.

The recurrence coefficients satisfy the so-called string equations, recall \eqref{eq:string}. In the notation of \eqref{eq:recurrence_monic}, these become
\begin{eqnarray}\label{eq:string2}
\beta_{n+1} &=\delta-\beta_n-\alpha_n^2, \label{eq:string_beta} \\
\alpha_{n+1} &=\frac{i(n+1)}{\beta_{n+1}}-\alpha_n. \label{eq:string_alpha}
\end{eqnarray}
These relations are very simple to evaluate. Unfortunately, these relations also quickly become numerically unstable for increasing $n$, with notable loss of accuracy beyond $n=20$.

On the other hand, assuming an accurate evaluation of the recurrence coefficients, the eigenvalue computation of the symmetric Jacobi matrix does appear to be accurate for larger values of $n$. Thus, one can evaluate the simple string equations in high-precision arithmetic, followed by an eigenvalue computation with standard floating point precision.

In practice, we found it convenient to approximate the recurrence coefficients a priori as a function of $\delta$, using Chebyshev expansions on an interval $[0,\delta_{\textrm{max}}]$. These expansions are computed once using the string equations \eqref{eq:string_beta}--\eqref{eq:string_alpha} in high-precision arithmetic, but the resulting expansions (for $\alpha_i(\delta)$ and $\beta_i(\delta)$) can be stored in standard precision. The quadrature rule for a given value of $\delta$ can then be computed on the fly using an efficient eigenvalue computation as described above. Assuming $N$ recurrence coefficients are precomputed, one can compute quadrature rules with any number of points between $1$ and $N$.

\section{Other phase functions with coalescing saddles}
\label{s:genericphase}

The derivation of uniform asymptotic expansions for oscillatory integrals with two coalescing saddle points is based on a smooth map to the cubic oscillator. The map was first investigated by Chester, Friedmann and Ursell \cite{chester1957}, but see also \cite[Chapter VII, \S4]{wong2001asymptotic}, \cite[\S23.4.1]{temme2014} and \cite[\S9.2]{bleistein1975asymptotic}. We include a brief, self-contained description here.

Let $g_\alpha(x)$ be an analytic function of $x$ and of the parameter $\alpha$, and such that $g$ has two stationary points $x_{1,2}(\alpha)$ that coalesce at a critical value, say, $\alpha^*$. We switch to a new variable $t$ such that
\begin{equation}\label{eq:chester}
 g_\alpha(x) = \frac13 t^3 - \delta(\alpha) t + A(\alpha).
\end{equation}
Here, $\delta(\alpha)$ and $A(\alpha)$ are determined such that the transformation is regular uniformly in $\alpha$, at least locally near $\alpha^*$. The saddle points $x_{1,2}(\alpha)$ of $g_\alpha$ must map to the saddle points $t_{1,2}(\alpha) = \mp \sqrt{\delta(\alpha)}$ of the cubic. By \eqref{eq:chester} this leads to the explicit expressions (see \cite[\S2]{chester1957} or \cite[(4.8) and (4.9)]{wong2001asymptotic})
\begin{equation}\label{eq:cubicmap_parameters}
 A(\alpha) = \frac12 g_\alpha(x_1) + \frac12 g_\alpha(x_2), \qquad \frac23 \delta^{3/2}(\alpha) = \frac12 g_\alpha(x_1) - \frac12 g_\alpha(x_2).
\end{equation}
It was shown originally in \cite{chester1957} that a branch of the map exists that is indeed analytic, i.e., for which $\frac{dx}{dt} \neq 0,\infty$.

After the change of variables from $x$ to $t$, the oscillatory integral involving a generic phase $g_\alpha(z)$ reduces to our canonical form,
\begin{equation}\label{eq:I_change}
 \int_{\Gamma_x} f(x) e^{i g_\alpha(x)} {\rm d}x = e^{i A(\alpha)} \int_{\Gamma_t} F(t) e^{i \left(\frac{t^3}{3} - \delta(\alpha)t\right)} {\rm d}t,
\end{equation}
with $F(t) = f(x) \frac{dx}{dt}$ and where the contour $\Gamma_t$ is the image of $\Gamma_x$.

It remains to compute the map and its jacobian. We can do so numerically using Newton-Raphson iterations, if a suitable initial guess can be determined. Using the formulas \eqref{eq:cubicmap_parameters}, an initial guess is readily obtained at least in a neighbourhood of the critical points. The jacobian of the map follows from differentiating the left and right hand sides of \eqref{eq:chester},
\begin{equation}\label{eq:jacobian}
 \frac{dx}{dt} = \frac{t^2 - \delta(\alpha)}{g_\alpha'(x)}.
\end{equation}
We illustrate the approach with a numerical example in the next section. An alternative numerical approach may be based on numerically integrating \eqref{eq:jacobian} starting from the critical points, but we found the simpler Newton-Raphson approach sufficient in our example.

Finally, we remark that \eqref{eq:cubicmap_parameters} can be seen as a generalization of \eqref{eq:change_of_variables}. Indeed, in the specific case where the oscillator is given by $g(x) = \omega \left(\frac{x^3}{3} - c x\right)$, \eqref{eq:cubicmap_parameters} gives rise to the substitutions $\delta = c \omega^{2/3}$ and $t = \omega^{1/3}x$, in agreement with \eqref{eq:change_of_variables}.

\section{Example: the computation of Chebyshev moments}
\label{sect_application}


\setlength{\figurewidth}{5.8cm}
\setlength{\figureheight}{5cm}

\begin{figure}[t]
\centering
\subfigure[$n=4$]{
\begin{small}
\input{relerror_chebyshev_n4_v2.tikz}
\end{small}
}
\subfigure[$n=8$]{
\begin{small}
\input{relerror_chebyshev_n8_v2.tikz}
\end{small}
}
\caption{Oscillatory moments of Chebyshev polynomials are computed using the UNSD quadrature rule, with polynomial degree $k$ and frequency $\omega$. Relative error is shown as a function of $k$, for several values of $m = \frac{k}{\omega}$. The known recurrence for the moments is unstable when $m \approx 1$, but UNSD is very accurate using only a very small number of points.}\label{fig:chebyshev}
\end{figure}
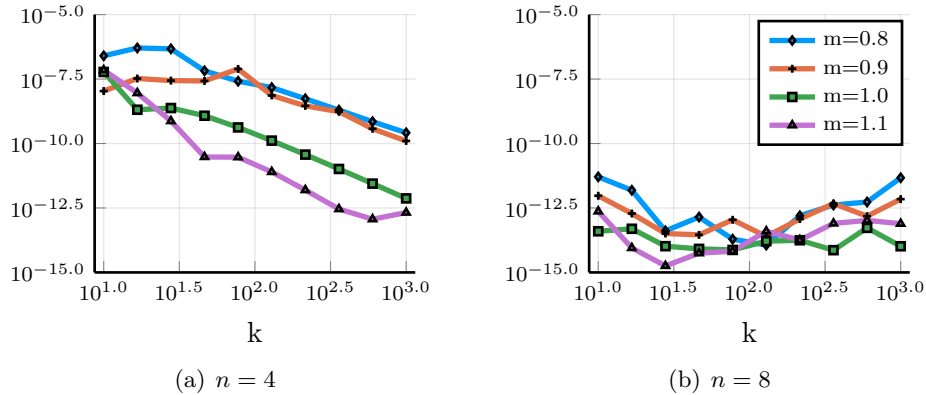

As an example of practical interest, we consider the evaluation of oscillatory moments of Chebyshev polynomials,
\[
 \int_{-1}^1 T_k(x)e^{i\omega x}dx
\]
where the polynomials have moderate to high degree $k$ and the frequency $\omega$ is large. These Chebyshev moments regularly appear in numerical methods for Fourier-type integrals of the form \eqref{eq:I} with $g(x)=x$, in which $f$ is approximated by a Chebyshev expansion \cite{dominguez2011filonclenshawcurtis,ledoux2012interpolatory,huybrechs2012superinterpolation}.

The moments satisfy a recursion, based on the recurrence relation of the Chebyshev polynomials themselves \cite{piessens1971oscillatory}. This recursion is numerically stable for small $k$ \cite{dominguez2011filonclenshawcurtis}. However, it becomes unstable once $k \sim \omega$. The remaining moments can be computed at once by solving a linear system of equations \cite{dominguez2011filonclenshawcurtis}. This system corresponds to following the recurrence in the backwards direction, starting from a moment of high degree that can be determined by other means. For example, in \cite{dominguez2011filonclenshawcurtis} it is determined asymptotically (assuming $k \approx 4\omega$). Since the system is tridiagonal, this is a very efficient way of computing moments for a wide range of degrees of the Chebyshev polynomial. However, it is less effective for the evaluation of just a few moments around the instability $k \sim \omega$. In our example, we will focus exactly on those integrals.


The resonance in the regime $k \sim \omega$ can be reformulated as a problem with coalescing saddle points. This is seen by rewriting the Chebyshev moment as follows, using well-known properties of the Chebyshev polynomials of the first kind:
\begin{align*}
\int_{-1}^1T_k(x) e^{i\omega x}dx &= \int_0^\pi \sin y \, T_k(\cos y)e^{i\omega \cos y} {\rm d}y
\\&=\int_0^\pi\sin y \cos (ky)e^{i\omega \cos y} {\rm d}y
\\&=\frac{1}{2}\int_0^\pi \sin y \, e^{i ky}e^{i\omega \cos y}dt+\frac{1}{2}\int_0^\pi\sin y \, e^{-iky}e^{i\omega \cos y} {\rm d}y \\
&= I_1 + I_2.
\end{align*}
The integral is written as the sum of two integrals of the form \eqref{eq:I}, with oscillatory factors
\[
e^{\pm iky}e^{i\omega \cos y}=e^{i\omega(\pm \mu y+\cos y)},
\]
where $\mu=\frac{k}{\omega}$.
It can be verified that the second oscillator does not have stationary points in the interval $[0,\pi]$. The first oscillator, $\mu y+\cos y$, has two stationary points
\[
 y_1 = \arcsin(\mu), \qquad y_2 = \pi - \arcsin(\mu).
\]
These points coalesce when $\mu = 1$. Thus, the second integral can be approximated by the regular NSD-method, and the first by the UNSD-algorithm. Note that the stationary points are complex-valued if $\mu > 1$.

In order to apply UNSD, following \S\ref{s:genericphase} the oscillator of the second integral is mapped to a cubic oscillator by a change of variables. The full parameter-dependent oscillator is
\[
 g(y) = ky + \omega \cos(y).
\]
We can compute explicitly that
\[
g(y_1) = g(\arcsin(\mu)) = k \arcsin(\mu) + \omega \cos(\arcsin(\mu)) = k \arcsin(\mu) + \omega \sqrt{1-\mu^2},
\]
and, similarly, that $g(y_2) = k \pi - g(y_1)$. Using~\eqref{eq:cubicmap_parameters} this yields
\[
 A = \frac12 g(y_1) + \frac12 g(y_2) = k \frac{\pi}{2}
\]
and
\[
 \delta = \left( k \arcsin(\mu) + \omega \sqrt{1-\mu^2} - k \frac{\pi}{2} \right)^{2/3} \left(\frac32\right)^{2/3}.
\]
Care has to be taken to select the right branch of the cubic root in the expression for $\delta$. In case $\mu < 1$, the standard branch is the appropriate one and $\delta$ is real and positive. In case $\mu > 1$, we chose the branch such that $\delta$ is real and negative.

Next, we have to map the quadrature points $t_{k,\delta}$ of the Gaussian rule back to the $y$-plane. As starting values for Newton-Raphson, we used a crude linear interpolation between the critical points,
\[
 y_{k,\delta}^{(0)} = y_1 + \frac{t_{k,\delta}-t_1}{t_2-t_1} (y_2 - y_1).
\]
That is, we assume that the location of $y_{k,\delta}$ relative to $y_1$ and $y_2$ is similar to the location of $t_{t,\delta}$, relative to $t_1$ and $t_2$. Finally, having computed the mapped quadrature points $y_{k,\delta}$, we can evaluate the jacobian $F'(t)$ using \eqref{eq:jacobian}.

The case of particular practical interest arises when $\mu = \frac{k}{\omega} \approx 1$. Figure \ref{fig:chebyshev} illustrates the accuracy of applying the UNSD scheme to integral $I_1$. It turns out that the method is very accurate precisely in the regime $\mu \approx 1$, for large $k$, using just $4$ quadrature points. Increasing the quadature points to $n=8$ leads to high relative accuracy even for moderately small values of $k$ (and correspondingly small values of $\omega = \frac{k}{\mu}$).

\section*{Acknowledgements}
All authors were supported by FWO Flanders project G.0641.11. In addition, Arno Kuijlaars is supported by a long term structural funding-Methusalem grant of the Flemish Government, and via support of FWO Flanders through projects G.0864.16 and EOS G0G9118N. Daan Huybrechs had additional support from FWO Flanders project G.A004.14N.

\bibliographystyle{abbrv}
\bibliography{cubicosc}

\end{document}

%% file: unsd_abs_error_v2.tikz
\begin{tikzpicture}[]
\begin{axis}[height = \figureheight, ylabel = {}, xmin = {0.8354841210835904}, xmax = {478.7643354384942}, ymax = {0.001}, ymode = {log}, xlabel = {$\omega$}, unbounded coords=jump,scaled x ticks = false,xlabel style = {font = {\fontsize{11 pt}{14.3 pt}\selectfont}, color = {rgb,1:red,0.00000000;green,0.00000000;blue,0.00000000}, draw opacity = 1.0, rotate = 0.0},log basis x=10,xmajorgrids = true,xtick = {1.0,3.162277660168379,10.0,31.622776601683793,100.0,316.22776601683796},xticklabels = {$10^{0.0}$,$10^{0.5}$,$10^{1.0}$,$10^{1.5}$,$10^{2.0}$,$10^{2.5}$},xtick align = inside,xticklabel style = {font = {\fontsize{8 pt}{10.4 pt}\selectfont}, color = {rgb,1:red,0.00000000;green,0.00000000;blue,0.00000000}, draw opacity = 1.0, rotate = 0.0},x grid style = {color = {rgb,1:red,0.00000000;green,0.00000000;blue,0.00000000},
draw opacity = 0.1,
line width = 0.5,
solid},axis x line* = left,x axis line style = {color = {rgb,1:red,0.00000000;green,0.00000000;blue,0.00000000},
draw opacity = 1.0,
line width = 1,
solid},scaled y ticks = false,ylabel style = {font = {\fontsize{11 pt}{14.3 pt}\selectfont}, color = {rgb,1:red,0.00000000;green,0.00000000;blue,0.00000000}, draw opacity = 1.0, rotate = 0.0},log basis y=10,ymajorgrids = true,ytick = {1.0e-18,1.0e-15,1.0e-12,1.0e-9,1.0e-6,0.001},yticklabels = {$10^{-18}$,$10^{-15}$,$10^{-12}$,$10^{-9}$,$10^{-6}$,$10^{-3}$},ytick align = inside,yticklabel style = {font = {\fontsize{8 pt}{10.4 pt}\selectfont}, color = {rgb,1:red,0.00000000;green,0.00000000;blue,0.00000000}, draw opacity = 1.0, rotate = 0.0},y grid style = {color = {rgb,1:red,0.00000000;green,0.00000000;blue,0.00000000},
draw opacity = 0.1,
line width = 0.5,
solid},axis y line* = left,y axis line style = {color = {rgb,1:red,0.00000000;green,0.00000000;blue,0.00000000},
draw opacity = 1.0,
line width = 1,
solid},    xshift = 0.0mm,
    yshift = 0.0mm,
    axis background/.style={fill={rgb,1:red,1.00000000;green,1.00000000;blue,1.00000000}}
,legend style = {color = {rgb,1:red,0.00000000;green,0.00000000;blue,0.00000000},
draw opacity = 1.0,
line width = 1,
solid,fill = {rgb,1:red,1.00000000;green,1.00000000;blue,1.00000000},font = {\fontsize{8 pt}{10.4 pt}\selectfont}},colorbar style={title=}, xmode = {log}, ymin = {1.0e-18}, width = \figurewidth]\addplot+ [color = {rgb,1:red,0.00000000;green,0.60560316;blue,0.97868012},
draw opacity = 1.0,
line width = 2,
solid,mark = diamond*,
mark size = 1.5,
mark options = {
    color = {rgb,1:red,0.00000000;green,0.00000000;blue,0.00000000}, draw opacity = 1.0,
    fill = {rgb,1:red,0.00000000;green,0.60560316;blue,0.97868012}, fill opacity = 1.0,
    line width = 1,
    rotate = 0,
    solid
}]coordinates {
(1.0, 0.00015696257686548254)
(1.9458877175763887, 5.5611592194582116e-6)
(3.786479009414647, 2.0972679215525153e-7)
(7.368062997280772, 8.324244293877198e-9)
(14.337423288737728, 3.436685380769483e-10)
(27.899015879248413, 1.4608092752928896e-11)
(54.28835233189811, 6.346352084218202e-13)
(105.63903801010007, 2.726597909670731e-14)
(205.56170656043895, 2.455766100420999e-15)
(399.9999999999999, 6.483277915462723e-16)
};
\addlegendentry{$c=0.001$}
\addplot+ [color = {rgb,1:red,0.88887350;green,0.43564919;blue,0.27812294},
draw opacity = 1.0,
line width = 2,
solid,mark = +,
mark size = 1.5,
mark options = {
    color = {rgb,1:red,0.00000000;green,0.00000000;blue,0.00000000}, draw opacity = 1.0,
    fill = {rgb,1:red,0.88887350;green,0.43564919;blue,0.27812294}, fill opacity = 1.0,
    line width = 1,
    rotate = 0,
    solid
}]coordinates {
(1.0, 0.0001671126957537572)
(1.9458877175763887, 6.184967707545397e-6)
(3.786479009414647, 2.499033726621753e-7)
(7.368062997280772, 1.1047398107423535e-8)
(14.337423288737728, 5.38865841197671e-10)
(27.899015879248413, 2.959049680980068e-11)
(54.28835233189811, 1.89187204836875e-12)
(105.63903801010007, 1.4920880198954041e-13)
(205.56170656043895, 1.5739734535245773e-14)
(399.9999999999999, 1.9271865356071313e-15)
};
\addlegendentry{$c=0.05$}
\addplot+ [color = {rgb,1:red,0.24222430;green,0.64327509;blue,0.30444865},
draw opacity = 1.0,
line width = 2,
solid,mark = square*,
mark size = 1.5,
mark options = {
    color = {rgb,1:red,0.00000000;green,0.00000000;blue,0.00000000}, draw opacity = 1.0,
    fill = {rgb,1:red,0.24222430;green,0.64327509;blue,0.30444865}, fill opacity = 1.0,
    line width = 1,
    rotate = 0,
    solid
}]coordinates {
(1.0, 0.00020169633808661547)
(1.9458877175763887, 8.489838915659575e-6)
(3.786479009414647, 4.1851787851988076e-7)
(7.368062997280772, 2.4963187239550854e-8)
(14.337423288737728, 1.878662114762725e-9)
(27.899015879248413, 1.8399637512078165e-10)
(54.28835233189811, 2.0869916997436647e-11)
(105.63903801010007, 6.400460136232822e-13)
(205.56170656043895, 2.6312569799020003e-14)
(399.9999999999999, 5.487466528429191e-15)
};
\addlegendentry{$c=0.2$}
\end{axis}

\end{tikzpicture}

%% file: nsd_vs_unsd_v2.tikz
\begin{tikzpicture}[]
\begin{axis}[height = \figureheight, ylabel = {}, xmin = {-0.05897}, xmax = {2.05997}, ymax = {0.06906972396892747}, ymode = {log}, xlabel = {c}, unbounded coords=jump,scaled x ticks = false,xlabel style = {font = {\fontsize{11 pt}{14.3 pt}\selectfont}, color = {rgb,1:red,0.00000000;green,0.00000000;blue,0.00000000}, draw opacity = 1.0, rotate = 0.0},xmajorgrids = true,xtick = {0.0,0.5,1.0,1.5,2.0},xticklabels = {$0.0$,$0.5$,$1.0$,$1.5$,$2.0$},xtick align = inside,xticklabel style = {font = {\fontsize{8 pt}{10.4 pt}\selectfont}, color = {rgb,1:red,0.00000000;green,0.00000000;blue,0.00000000}, draw opacity = 1.0, rotate = 0.0},x grid style = {color = {rgb,1:red,0.00000000;green,0.00000000;blue,0.00000000},
draw opacity = 0.1,
line width = 0.5,
solid},axis x line* = left,x axis line style = {color = {rgb,1:red,0.00000000;green,0.00000000;blue,0.00000000},
draw opacity = 1.0,
line width = 1,
solid},scaled y ticks = false,ylabel style = {font = {\fontsize{11 pt}{14.3 pt}\selectfont}, color = {rgb,1:red,0.00000000;green,0.00000000;blue,0.00000000}, draw opacity = 1.0, rotate = 0.0},log basis y=10,ymajorgrids = true,ytick = {1.0e-15,3.162277660168379e-13,1.0e-10,3.162277660168379e-8,1.0e-5,0.0031622776601683794},yticklabels = {$10^{-15.0}$,$10^{-12.5}$,$10^{-10.0}$,$10^{-7.5}$,$10^{-5.0}$,$10^{-2.5}$},ytick align = inside,yticklabel style = {font = {\fontsize{8 pt}{10.4 pt}\selectfont}, color = {rgb,1:red,0.00000000;green,0.00000000;blue,0.00000000}, draw opacity = 1.0, rotate = 0.0},y grid style = {color = {rgb,1:red,0.00000000;green,0.00000000;blue,0.00000000},
draw opacity = 0.1,
line width = 0.5,
solid},axis y line* = left,y axis line style = {color = {rgb,1:red,0.00000000;green,0.00000000;blue,0.00000000},
draw opacity = 1.0,
line width = 1,
solid},    xshift = 0.0mm,
    yshift = 0.0mm,
    axis background/.style={fill={rgb,1:red,1.00000000;green,1.00000000;blue,1.00000000}}
,legend style = {color = {rgb,1:red,0.00000000;green,0.00000000;blue,0.00000000},
draw opacity = 1.0,
line width = 1,
solid,fill = {rgb,1:red,1.00000000;green,1.00000000;blue,1.00000000},font = {\fontsize{8 pt}{10.4 pt}\selectfont}},colorbar style={title=}, ymin = {9.115986127960835e-17}, width = \figurewidth]\addplot+ [color = {rgb,1:red,0.00000000;green,0.60560316;blue,0.97868012},
draw opacity = 1.0,
line width = 2,
solid,mark = diamond*,
mark size = 1.5,
mark options = {
    color = {rgb,1:red,0.00000000;green,0.00000000;blue,0.00000000}, draw opacity = 1.0,
    fill = {rgb,1:red,0.00000000;green,0.60560316;blue,0.97868012}, fill opacity = 1.0,
    line width = 1,
    rotate = 0,
    solid
}]coordinates {
(0.001, 0.026192095577834584)
(0.05225641025641026, 0.0010663610856990746)
(0.10351282051282051, 2.65270965736919e-5)
(0.15476923076923077, 1.973014579491339e-5)
(0.20602564102564103, 4.9046051542323e-6)
(0.2572820512820513, 1.0574247872319198e-6)
(0.30853846153846154, 2.283581359563414e-7)
(0.3597948717948718, 5.203471031833832e-8)
(0.41105128205128205, 8.39935931917788e-9)
(0.4623076923076923, 6.750883935204341e-9)
(0.5135641025641026, 1.1935929244977463e-8)
(0.5648205128205128, 9.11495488808978e-9)
(0.6160769230769231, 6.080914946837081e-11)
(0.6673333333333333, 8.514898605649408e-9)
(0.7185897435897436, 7.856621783708505e-9)
(0.7698461538461538, 1.884425317178048e-9)
(0.8211025641025641, 6.975324401418427e-9)
(0.8723589743589744, 6.988327000571562e-10)
(0.9236153846153846, 5.411396997168083e-9)
(0.9748717948717949, 2.6312534928685242e-9)
(1.026128205128205, 2.647224883745027e-9)
(1.0773846153846154, 4.239413653679704e-9)
(1.1286410256410255, 2.2579782488225253e-9)
(1.179897435897436, 6.024104410065625e-10)
(1.231153846153846, 2.503699025879785e-9)
(1.2824102564102564, 3.1389706944829987e-9)
(1.3336666666666666, 2.9590371264386057e-9)
(1.384923076923077, 2.484317862540518e-9)
(1.436179487179487, 2.0464331340619593e-9)
(1.4874358974358974, 1.7875391167295047e-9)
(1.5386923076923076, 1.7310173572317277e-9)
(1.589948717948718, 1.832037938021765e-9)
(1.641205128205128, 1.9926184036186203e-9)
(1.6924615384615385, 2.0590517069153957e-9)
(1.7437179487179486, 1.8386993594359508e-9)
(1.794974358974359, 1.171505475961388e-9)
(1.8462307692307691, 8.312851995353686e-11)
(1.8974871794871795, 1.0593992627598315e-9)
(1.9487435897435896, 1.593384035929339e-9)
(2.0, 9.9342441217824e-10)
};
\addlegendentry{NSD, $\omega=100$}
\addplot+ [color = {rgb,1:red,0.88887350;green,0.43564919;blue,0.27812294},
draw opacity = 1.0,
line width = 2,
solid,mark = +,
mark size = 1.5,
mark options = {
    color = {rgb,1:red,0.00000000;green,0.00000000;blue,0.00000000}, draw opacity = 1.0,
    fill = {rgb,1:red,0.88887350;green,0.43564919;blue,0.27812294}, fill opacity = 1.0,
    line width = 1,
    rotate = 0,
    solid
}]coordinates {
(0.001, 0.002296367789823818)
(0.05225641025641026, 2.1739113412189148e-7)
(0.10351282051282051, 8.529900230591644e-9)
(0.15476923076923077, 9.715707682825122e-10)
(0.20602564102564103, 1.7116952300151274e-10)
(0.2572820512820513, 3.8689468297305586e-11)
(0.30853846153846154, 3.399211587548581e-12)
(0.3597948717948718, 2.7531727664359366e-12)
(0.41105128205128205, 4.262564412011263e-13)
(0.4623076923076923, 1.879251040966824e-13)
(0.5135641025641026, 1.7953001940281232e-13)
(0.5648205128205128, 2.0766204931676292e-13)
(0.6160769230769231, 3.386513047422638e-13)
(0.6673333333333333, 3.23807207472474e-13)
(0.7185897435897436, 1.3624986559274238e-14)
(0.7698461538461538, 8.353695108708069e-14)
(0.8211025641025641, 1.9628784904927308e-13)
(0.8723589743589744, 1.619065141454923e-13)
(0.9236153846153846, 6.344336739180011e-14)
(0.9748717948717949, 1.2138127822361275e-13)
(1.026128205128205, 1.3997181869876047e-13)
(1.0773846153846154, 4.1513587729875695e-14)
(1.1286410256410255, 8.837243153824972e-14)
(1.179897435897436, 4.5979021184571234e-14)
(1.231153846153846, 7.664793543315697e-14)
(1.2824102564102564, 2.4613853099175788e-14)
(1.3336666666666666, 6.58860148761593e-14)
(1.384923076923077, 6.180060363586664e-14)
(1.436179487179487, 1.1335236780559126e-13)
(1.4874358974358974, 2.0660979232043502e-14)
(1.5386923076923076, 6.593369702886845e-15)
(1.589948717948718, 6.424339665768679e-14)
(1.641205128205128, 2.60320049836757e-14)
(1.6924615384615385, 9.312686719482733e-14)
(1.7437179487179486, 8.013034691047801e-14)
(1.794974358974359, 7.797099776081767e-14)
(1.8462307692307691, 4.922409764982578e-14)
(1.8974871794871795, 5.565048354323138e-14)
(1.9487435897435896, 3.476997335815188e-14)
(2.0, 7.661296061464333e-14)
};
\addlegendentry{NSD, $\omega=1000$}
\addplot+ [color = {rgb,1:red,0.24222430;green,0.64327509;blue,0.30444865},
draw opacity = 1.0,
line width = 2,
solid,mark = square*,
mark size = 1.5,
mark options = {
    color = {rgb,1:red,0.00000000;green,0.00000000;blue,0.00000000}, draw opacity = 1.0,
    fill = {rgb,1:red,0.24222430;green,0.64327509;blue,0.30444865}, fill opacity = 1.0,
    line width = 1,
    rotate = 0,
    solid
}]coordinates {
(0.001, 1.3130594571788408e-14)
(0.05225641025641026, 1.317796226433564e-13)
(0.10351282051282051, 5.89909439615149e-13)
(0.15476923076923077, 1.874348052023019e-12)
(0.20602564102564103, 2.7504110175695848e-12)
(0.2572820512820513, 4.192774704811901e-12)
(0.30853846153846154, 4.348802188944983e-12)
(0.3597948717948718, 4.642453121217396e-12)
(0.41105128205128205, 6.227352019362138e-12)
(0.4623076923076923, 9.402839886848643e-12)
(0.5135641025641026, 1.2819911925994191e-11)
(0.5648205128205128, 1.202157123959885e-11)
(0.6160769230769231, 1.7631454816562921e-12)
(0.6673333333333333, 1.557966257483345e-11)
(0.7185897435897436, 2.1463872393520286e-11)
(0.7698461538461538, 4.435854478555288e-12)
(0.8211025641025641, 3.0310424414542406e-11)
(0.8723589743589744, 1.707037887560078e-12)
(0.9236153846153846, 3.6870922988130347e-11)
(0.9748717948717949, 1.9197199415208935e-11)
(1.026128205128205, 2.7816971002571976e-11)
(1.0773846153846154, 4.902902785626393e-11)
(1.1286410256410255, 2.82843748991475e-11)
(1.179897435897436, 1.1862962856090154e-11)
(1.231153846153846, 4.719827307174723e-11)
(1.2824102564102564, 6.549089050600547e-11)
(1.3336666666666666, 6.85739799283306e-11)
(1.384923076923077, 6.397732359181319e-11)
(1.436179487179487, 5.859795981789643e-11)
(1.4874358974358974, 5.7002735916033975e-11)
(1.5386923076923076, 6.163289323576308e-11)
(1.589948717948718, 7.292288996100357e-11)
(1.641205128205128, 8.860934227039323e-11)
(1.6924615384615385, 1.021343581848315e-10)
(1.7437179487179486, 1.0162631847229258e-10)
(1.794974358974359, 7.228340152266572e-11)
(1.8462307692307691, 7.435138612699622e-12)
(1.8974871794871795, 7.434367011804276e-11)
(1.9487435897435896, 1.2500850625623017e-10)
(2.0, 8.695234812164361e-11)
};
\addlegendentry{UNSD, $\omega=100$}
\addplot+ [color = {rgb,1:red,0.76444018;green,0.44411178;blue,0.82429754},
draw opacity = 1.0,
line width = 2,
solid,mark = triangle*,
mark size = 1.5,
mark options = {
    color = {rgb,1:red,0.00000000;green,0.00000000;blue,0.00000000}, draw opacity = 1.0,
    fill = {rgb,1:red,0.76444018;green,0.44411178;blue,0.82429754}, fill opacity = 1.0,
    line width = 1,
    rotate = 0,
    solid
}]coordinates {
(0.001, 5.459289433711218e-15)
(0.05225641025641026, 2.7768950580485576e-15)
(0.10351282051282051, 4.638171781942259e-16)
(0.15476923076923077, 1.5105485947802026e-15)
(0.20602564102564103, 1.852153913869558e-15)
(0.2572820512820513, 1.6558144943186592e-15)
(0.30853846153846154, 2.102086802081332e-15)
(0.3597948717948718, 4.098489506328572e-15)
(0.41105128205128205, 1.6677758164491582e-15)
(0.4623076923076923, 2.9767026697855666e-15)
(0.5135641025641026, 6.111054354846122e-16)
(0.5648205128205128, 2.4039261909828423e-16)
(0.6160769230769231, 7.145138161462101e-16)
(0.6673333333333333, 9.365981374399247e-16)
(0.7185897435897436, 1.9192505095162924e-14)
(0.7698461538461538, 2.8637567531594952e-15)
(0.8211025641025641, 3.941517733646298e-16)
(0.8723589743589744, 3.5005602214728766e-16)
(0.9236153846153846, 2.5424626160621296e-15)
(0.9748717948717949, 1.1004756405964171e-13)
(1.026128205128205, 2.4532694666933986e-15)
(1.0773846153846154, 4.177084110315971e-15)
(1.1286410256410255, 2.346085487358512e-15)
(1.179897435897436, 1.583197107659825e-15)
(1.231153846153846, 1.6950546136079997e-15)
(1.2824102564102564, 5.1956378454197384e-15)
(1.3336666666666666, 3.2023078813576262e-15)
(1.384923076923077, 2.2727630679249106e-15)
(1.436179487179487, 9.27712785410655e-14)
(1.4874358974358974, 3.083404609412573e-15)
(1.5386923076923076, 1.9068726855044305e-15)
(1.589948717948718, 2.440506349547915e-15)
(1.641205128205128, 4.279340255746748e-15)
(1.6924615384615385, 6.158988339959833e-14)
(1.7437179487179486, 4.1878567311719365e-15)
(1.794974358974359, 5.732502472993625e-15)
(1.8462307692307691, 3.6933119807295785e-15)
(1.8974871794871795, 8.40065276929936e-15)
(1.9487435897435896, 7.955976787157009e-15)
(2.0, 1.6954503148981027e-14)
};
\addlegendentry{UNSD, $\omega=1000$}
\end{axis}

\end{tikzpicture}

%% file: delta_points.tikz
\begin{tikzpicture}[]
\begin{axis}[height = \figureheight, ylabel = {}, xmin = {-4.772862382137342}, xmax = {4.772862382137335}, ymax = {2.292367419909967}, xlabel = {}, unbounded coords=jump,scaled x ticks = false,xticklabel style={rotate = 0},xmajorgrids = true,xtick = {-4.0,-2.0,0.0,2.0,4.0},xticklabels = {$-4$,$-2$,$0$,$2$,$4$},xtick align = inside,axis lines* = left,scaled y ticks = false,yticklabel style={rotate = 0},ymajorgrids = true,ytick = {-1.0,0.0,1.0,2.0},yticklabels = {$-1$,$0$,$1$,$2$},ytick align = inside,axis lines* = left,    xshift = 0.0mm,
    yshift = 0.0mm,
    axis background/.style={fill={rgb,1:red,1.00000000;green,1.00000000;blue,1.00000000}}
, ymin = {-1.0174905113928054}, width = \figurewidth]\addplot+[draw=none, color = {rgb,1:red,0.00000000;green,0.60560316;blue,0.97868012},
draw opacity=1.0,
line width=0,
solid,mark = *,
mark size = 2.5,
mark options = {
    color = {rgb,1:red,0.00000000;green,0.00000000;blue,0.00000000}, draw opacity = 1.0,
    fill = {rgb,1:red,0.00000000;green,0.60560316;blue,0.97868012}, fill opacity = 1.0,
    line width = 1,
    rotate = 0,
    solid
},forget plot] coordinates {
(-2.872200228101498, 2.1986921954391336)
(-2.388301590457297, 2.017373351431133)
(-1.9492155807876719, 1.8725888373752413)
(-1.5218843065527712, 1.7533945897094814)
(-1.0938726927841664, 1.6591140060790444)
(-0.6602369751457221, 1.592676063114306)
(-0.220866957347741, 1.5580564466045763)
(0.22086695734773593, 1.5580564466045943)
(0.6602369751457274, 1.5926760631142987)
(1.0938726927841642, 1.6591140060790446)
(1.5218843065527639, 1.753394589709477)
(1.949215580787667, 1.8725888373752386)
(2.388301590457299, 2.017373351431131)
(2.8722002281015055, 2.1986921954391314)
};
\addplot+[draw=none, color = {rgb,1:red,0.88887350;green,0.43564919;blue,0.27812294},
draw opacity=1.0,
line width=0,
solid,mark = diamond*,
mark size = 3.0,
mark options = {
    color = {rgb,1:red,0.00000000;green,0.00000000;blue,0.00000000}, draw opacity = 1.0,
    fill = {rgb,1:red,0.88887350;green,0.43564919;blue,0.27812294}, fill opacity = 1.0,
    line width = 1,
    rotate = 0,
    solid
},forget plot] coordinates {
(-3.0870797274051585, 1.8857191503743405)
(-2.601918775136034, 1.6632088065564297)
(-2.1545269742071333, 1.475442667719378)
(-1.708980905242275, 1.3103622847008598)
(-1.248675498405578, 1.1693478478722357)
(-0.7648811431419795, 1.0615369265715653)
(-0.25823721610424377, 1.0014253793514156)
(0.2582372161042484, 1.001425379351407)
(0.7648811431419767, 1.0615369265715702)
(1.2486754984055755, 1.169347847872244)
(1.7089809052422849, 1.3103622847008625)
(2.1545269742071285, 1.4754426677193797)
(2.601918775136039, 1.6632088065564332)
(3.087079727405159, 1.8857191503743438)
};
\addplot+[draw=none, color = {rgb,1:red,0.24222430;green,0.64327509;blue,0.30444865},
draw opacity=1.0,
line width=0,
solid,mark = +,
mark size = 3.0,
mark options = {
    color = {rgb,1:red,0.00000000;green,0.00000000;blue,0.00000000}, draw opacity = 1.0,
    fill = {rgb,1:red,0.24222430;green,0.64327509;blue,0.30444865}, fill opacity = 1.0,
    line width = 1,
    rotate = 0,
    solid
},forget plot] coordinates {
(3.349500430361023, 1.578981158614562)
(2.8775434182445556, 1.3096255790648816)
(2.4384773283918926, 1.0676439065267238)
(1.9927477483099438, 0.8355115548774927)
(1.5140760604383785, 0.610150508366164)
(0.9732526936601166, 0.4026552849841147)
(0.3428767239122525, 0.2577678204088479)
(-0.34287672391231233, 0.2577678204086882)
(-0.973252693660053, 0.40265528498417313)
(-3.349500430361022, 1.578981158614565)
(-1.5140760604383963, 0.6101505083661197)
(-2.877543418244548, 1.3096255790648799)
(-1.9927477483099334, 0.8355115548775169)
(-2.4384773283918912, 1.0676439065267056)
};
\addplot+[draw=none, color = {rgb,1:red,0.76444018;green,0.44411178;blue,0.82429754},
draw opacity=1.0,
line width=0,
solid,mark = square*,
mark size = 2.5,
mark options = {
    color = {rgb,1:red,0.00000000;green,0.00000000;blue,0.00000000}, draw opacity = 1.0,
    fill = {rgb,1:red,0.76444018;green,0.44411178;blue,0.82429754}, fill opacity = 1.0,
    line width = 1,
    rotate = 0,
    solid
},forget plot] coordinates {
(3.6612399445844783, 1.2953216792094302)
(3.2216460497128088, 0.9827044244046145)
(2.818185383119958, 0.6880978394240281)
(2.4146334160061693, 0.3843916213230487)
(1.9872327891901578, 0.05033664403276669)
(1.503392027902972, -0.3487362693595981)
(0.879892886062967, -0.9238152869219722)
(-0.8798928860629494, -0.9238152869219647)
(-3.6612399445844783, 1.295321679209433)
(-1.503392027903008, -0.34873626935960333)
(-3.2216460497128128, 0.9827044244046246)
(-1.9872327891901413, 0.05033664403279811)
(-2.8181853831199533, 0.6880978394240387)
(-2.414633416006159, 0.38439162132303456)
};
\addplot+[draw=none, color = {rgb,1:red,0.67554396;green,0.55566233;blue,0.09423434},
draw opacity=1.0,
line width=0,
solid,mark = triangle*,
mark size = 3.0,
mark options = {
    color = {rgb,1:red,0.00000000;green,0.00000000;blue,0.00000000}, draw opacity = 1.0,
    fill = {rgb,1:red,0.67554396;green,0.55566233;blue,0.09423434}, fill opacity = 1.0,
    line width = 1,
    rotate = 0,
    solid
},forget plot] coordinates {
(3.9941995485647954, 1.1982581099692882)
(3.5875640294865287, 0.8869145248893621)
(3.220160112606772, 0.5956212260492204)
(2.8600007591815273, 0.29853117845981186)
(2.4894145261299925, -0.022447268156466305)
(2.0893714406487076, -0.39236849718379063)
(1.6227140117616263, -0.8697702291644127)
(-1.6227140117616279, -0.8697702291644078)
(-3.9941995485648043, 1.1982581099692915)
(-2.0893714406487147, -0.3923684971837863)
(-3.5875640294865416, 0.8869145248893658)
(-2.4894145261299916, -0.022447268156476376)
(-3.220160112606776, 0.5956212260492187)
(-2.860000759181517, 0.29853117845980937)
};
\addplot+[draw=none, color = {rgb,1:red,0.00000048;green,0.66575898;blue,0.68099695},
draw opacity=1.0,
line width=0,
solid,mark = pentagon*,
mark size = 3.0,
mark options = {
    color = {rgb,1:red,0.00000000;green,0.00000000;blue,0.00000000}, draw opacity = 1.0,
    fill = {rgb,1:red,0.00000048;green,0.66575898;blue,0.68099695}, fill opacity = 1.0,
    line width = 1,
    rotate = 0,
    solid
},forget plot] coordinates {
(4.502700360506919, 1.0587771012392102)
(4.133151751612235, 0.7517300986563201)
(3.8037616636991816, 0.46744606766056807)
(3.4856397095486455, 0.18141173629528817)
(3.16386421663279, -0.12191812280545768)
(2.823399237835479, -0.462105251631937)
(2.4349603047677584, -0.8821411653847362)
(-2.434960304767757, -0.8821411653847383)
(-4.5027003605069265, 1.0587771012392109)
(-2.823399237835495, -0.46210525163194116)
(-3.163864216632789, -0.1219181228054599)
(-4.133151751612239, 0.7517300986563212)
(-3.8037616636991936, 0.46744606766056934)
(-3.485639709548655, 0.18141173629529173)
};
\end{axis}

\end{tikzpicture}

%% file: unsd_hermite_points.tikz
\begin{tikzpicture}[]
\begin{axis}[height = \figureheight, legend style={at={(0.3,0.8)},anchor=west}, ylabel = {}, xmin = {-4.502146444379752}, xmax = {4.502146444379749}, ymax = {1.1757871656163537}, xlabel = {}, unbounded coords=jump,scaled x ticks = false,xticklabel style={rotate = 0},xmajorgrids = true,xtick = {-4.0,-2.0,0.0,2.0,4.0},xticklabels = {$-4$,$-2$,$0$,$2$,$4$},xtick align = inside,axis lines* = left,scaled y ticks = false,yticklabel style={rotate = 0},ymajorgrids = true,ytick = {-1.0,-0.5,0.0,0.5,1.0},yticklabels = {$-1.0$,$-0.5$,$0.0$,$0.5$,$1.0$},ytick align = inside,axis lines* = left,    xshift = 0.0mm,
    yshift = 0.0mm,
    axis background/.style={fill={rgb,1:red,1.00000000;green,1.00000000;blue,1.00000000}}
, ymin = {-1.2550562067939208}, width = \figurewidth]\addplot+[draw=none, color = {rgb,1:red,0.00000000;green,0.60560316;blue,0.97868012},
draw opacity=1.0,
line width=0,
solid,mark = triangle*,
mark size = 3.0,
mark options = {
    color = {rgb,1:red,0.00000000;green,0.00000000;blue,0.00000000}, draw opacity = 1.0,
    fill = {rgb,1:red,0.00000000;green,0.60560316;blue,0.97868012}, fill opacity = 1.0,
    line width = 1,
    rotate = 180,
    solid
}] coordinates {
(4.247307966395989, 1.1069897116802054)
(3.861750330241342, 0.7941718702202738)
(3.516837501788004, 0.5028504244085757)
(3.182604891528303, 0.20784475290310195)
(2.843452128516614, -0.10732849958959197)
(2.4835046520364497, -0.4640186019480636)
(2.0718537681202944, -0.9099679049304426)
(-2.071853768120329, -0.9099679049304087)
(-4.247307966395993, 1.1069897116802139)
(-2.483504652036414, -0.46401860194809597)
(-3.8617503302413274, 0.7941718702202445)
(-3.5168375017879745, 0.5028504244085478)
(-2.8434521285166294, -0.10732849958959391)
(-3.182604891528285, 0.2078447529031133)
};
\addlegendentry{unsd}
\addplot+[draw=none, color = {rgb,1:red,0.88887350;green,0.43564919;blue,0.27812294},
draw opacity=1.0,
line width=0,
solid,mark = x,
mark size = 3.0,
mark options = {
    color = {rgb,1:red,0.00000000;green,0.00000000;blue,0.00000000}, draw opacity = 1.0,
    fill = {rgb,1:red,0.88887350;green,0.43564919;blue,0.27812294}, fill opacity = 1.0,
    line width = 1,
    rotate = 0,
    solid
}] coordinates {
(2.130828870354901, -1.186258752857781)
(2.5019363351883177, -0.7160836862148421)
(2.8382228628701935, -0.3361616503921381)
(3.162277660168379, 0.0)
(3.4864364504232572, 0.3139516389269259)
(3.824453802473564, 0.6227876816025906)
(4.20520791113596, 0.952818550526416)
(-2.130828870354901, -1.186258752857781)
(-2.5019363351883177, -0.7160836862148421)
(-2.8382228628701935, -0.3361616503921381)
(-3.162277660168379, 0.0)
(-3.4864364504232572, 0.3139516389269259)
(-3.824453802473564, 0.6227876816025906)
(-4.20520791113596, 0.952818550526416)
};
\addlegendentry{hermite}
\end{axis}

\end{tikzpicture}

%% file: error_even_n_v2.tikz
\begin{tikzpicture}[]
\begin{axis}[height = \figureheight, ylabel = {}, xmin = {-0.29897}, xmax = {10.29997}, ymax = {7.260219775302302e-6}, ymode = {log}, xlabel = {$c$}, unbounded coords=jump,scaled x ticks = false,xlabel style = {font = {\fontsize{11 pt}{14.3 pt}\selectfont}, color = {rgb,1:red,0.00000000;green,0.00000000;blue,0.00000000}, draw opacity = 1.0, rotate = 0.0},xmajorgrids = true,xtick = {0.0,2.5,5.0,7.5,10.0},xticklabels = {$0.0$,$2.5$,$5.0$,$7.5$,$10.0$},xtick align = inside,xticklabel style = {font = {\fontsize{8 pt}{10.4 pt}\selectfont}, color = {rgb,1:red,0.00000000;green,0.00000000;blue,0.00000000}, draw opacity = 1.0, rotate = 0.0},x grid style = {color = {rgb,1:red,0.00000000;green,0.00000000;blue,0.00000000},
draw opacity = 0.1,
line width = 0.5,
solid},axis x line* = left,x axis line style = {color = {rgb,1:red,0.00000000;green,0.00000000;blue,0.00000000},
draw opacity = 1.0,
line width = 1,
solid},scaled y ticks = false,ylabel style = {font = {\fontsize{11 pt}{14.3 pt}\selectfont}, color = {rgb,1:red,0.00000000;green,0.00000000;blue,0.00000000}, draw opacity = 1.0, rotate = 0.0},log basis y=10,ymajorgrids = true,ytick = {1.0e-12,1.0e-10,1.0e-8,1.0e-6},yticklabels = {$10^{-12}$,$10^{-10}$,$10^{-8}$,$10^{-6}$},ytick align = inside,yticklabel style = {font = {\fontsize{8 pt}{10.4 pt}\selectfont}, color = {rgb,1:red,0.00000000;green,0.00000000;blue,0.00000000}, draw opacity = 1.0, rotate = 0.0},y grid style = {color = {rgb,1:red,0.00000000;green,0.00000000;blue,0.00000000},
draw opacity = 0.1,
line width = 0.5,
solid},axis y line* = left,y axis line style = {color = {rgb,1:red,0.00000000;green,0.00000000;blue,0.00000000},
draw opacity = 1.0,
line width = 1,
solid},    xshift = 0.0mm,
    yshift = 0.0mm,
    axis background/.style={fill={rgb,1:red,1.00000000;green,1.00000000;blue,1.00000000}}
,legend style = {color = {rgb,1:red,0.00000000;green,0.00000000;blue,0.00000000},
draw opacity = 1.0,
line width = 1,
solid,fill = {rgb,1:red,1.00000000;green,1.00000000;blue,1.00000000},font = {\fontsize{8 pt}{10.4 pt}\selectfont}},colorbar style={title=}, ymin = {4.902845617584364e-14}, width = \figurewidth]\addplot+ [color = {rgb,1:red,0.00000000;green,0.60560316;blue,0.97868012},
draw opacity = 1.0,
line width = 2,
solid,mark = none,
mark size = 2.0,
mark options = {
    color = {rgb,1:red,0.00000000;green,0.00000000;blue,0.00000000}, draw opacity = 1.0,
    fill = {rgb,1:red,0.00000000;green,0.60560316;blue,0.97868012}, fill opacity = 1.0,
    line width = 1,
    rotate = 0,
    solid
}]coordinates {
(0.001, 1.6829026705840605e-8)
(0.102, 1.9741139125726956e-8)
(0.203, 2.3108959399749552e-8)
(0.304, 2.699468627534797e-8)
(0.405, 3.146790698458072e-8)
(0.506, 3.660539778269354e-8)
(0.607, 4.249288165138686e-8)
(0.708, 4.922410304087401e-8)
(0.809, 5.690231514689978e-8)
(0.91, 6.564091848937672e-8)
(1.011, 7.556485764405114e-8)
(1.112, 8.680856981168257e-8)
(1.213, 9.95196602259358e-8)
(1.314, 1.1385862247918503e-7)
(1.415, 1.2999843373612008e-7)
(1.516, 1.4812666451566157e-7)
(1.617, 1.6844637268240688e-7)
(1.718, 1.9117536519981956e-7)
(1.819, 2.1655049965752762e-7)
(1.92, 2.4482597843162287e-7)
(2.021, 2.7627627408739206e-7)
(2.122, 3.111984954579261e-7)
(2.223, 3.499130417133883e-7)
(2.324, 3.927671990661185e-7)
(2.425, 4.401404353494518e-7)
(2.526, 4.924372576445557e-7)
(2.627, 5.500982847803361e-7)
(2.728, 6.13591357307137e-7)
(2.829, 6.834041526021002e-7)
(2.93, 7.600219342351818e-7)
(3.031, 8.438846088223252e-7)
(3.132, 9.353034598853301e-7)
(3.233, 1.0343128007549292e-6)
(3.334, 1.1404160968604627e-6)
(3.435, 1.2521663048788837e-6)
(3.536, 1.3665214919433839e-6)
(3.637, 1.4779245446366572e-6)
(3.738, 1.5771770086611032e-6)
(3.839, 1.6504381472941355e-6)
(3.94, 1.6791419938478915e-6)
(4.041, 1.6420008498565887e-6)
(4.142, 1.5201953463492748e-6)
(4.243, 1.3049439638817184e-6)
(4.344, 1.0049774242365036e-6)
(4.445, 6.533229917524669e-7)
(4.546, 3.4770596055525815e-7)
(4.647, 4.1240492645064026e-7)
(4.748, 7.482027886323907e-7)
(4.849, 1.105365486657358e-6)
(4.95, 1.4398766945537841e-6)
(5.051, 1.7413038766415458e-6)
(5.152, 2.002412969885729e-6)
(5.253, 2.211193819676486e-6)
(5.354, 2.3474513788024383e-6)
(5.455, 2.3831357216737304e-6)
(5.556, 2.287980270313504e-6)
(5.657, 2.041045387471641e-6)
(5.758, 1.6451485506531355e-6)
(5.859, 1.1401540544946949e-6)
(5.96, 6.395691649654223e-7)
(6.061, 5.558901605962922e-7)
(6.162, 1.0079946646244397e-6)
(6.263, 1.535965284679413e-6)
(6.364, 2.0214601303310184e-6)
(6.465, 2.4325544613994367e-6)
(6.566, 2.745486438679235e-6)
(6.667, 2.929422098210163e-6)
(6.768, 2.945679074590755e-6)
(6.869, 2.7581573339936343e-6)
(6.97, 2.3528629393434705e-6)
(7.071, 1.759708839465499e-6)
(7.172, 1.0815289037396826e-6)
(7.273, 6.726937974351838e-7)
(7.374, 1.0801120616909198e-6)
(7.475, 1.7559172283977968e-6)
(7.576, 2.391977665423091e-6)
(7.677, 2.9156429765791746e-6)
(7.778, 3.2838349500632807e-6)
(7.879, 3.4500958144489017e-6)
(7.98, 3.3656070263527167e-6)
(8.081, 2.9988853390479654e-6)
(8.182, 2.365343511518117e-6)
(8.283, 1.560652106992414e-6)
(8.384, 8.893343338062994e-7)
(8.485, 1.1271743374032452e-6)
(8.586, 1.9165515099915543e-6)
(8.687, 2.6944161076795564e-6)
(8.788, 3.3252954394795285e-6)
(8.889, 3.7425861688863503e-6)
(8.99, 3.884379619371602e-6)
(9.091, 3.6954080162611793e-6)
(9.192, 3.1560701643748175e-6)
(9.293, 2.3221983963138626e-6)
(9.394, 1.390097669535096e-6)
(9.495, 1.0350546379457978e-6)
(9.596, 1.7634809151325433e-6)
(9.697, 2.696384778668917e-6)
(9.798, 3.4924831104672047e-6)
(9.899, 4.041379457132599e-6)
(10.0, 4.262939386800252e-6)
};
\addlegendentry{$n=6$}
\addplot+ [color = {rgb,1:red,0.88887350;green,0.43564919;blue,0.27812294},
draw opacity = 1.0,
line width = 2,
dashed,mark = none,
mark size = 2.0,
mark options = {
    color = {rgb,1:red,0.00000000;green,0.00000000;blue,0.00000000}, draw opacity = 1.0,
    fill = {rgb,1:red,0.88887350;green,0.43564919;blue,0.27812294}, fill opacity = 1.0,
    line width = 1,
    rotate = 0,
    solid
}]coordinates {
(0.001, 1.3813026518282948e-12)
(0.102, 2.233490170216201e-12)
(0.203, 2.2361733706246177e-12)
(0.304, 3.0269702362413535e-12)
(0.405, 3.017634937627545e-12)
(0.506, 3.6560288903384395e-12)
(0.607, 4.7985319754168365e-12)
(0.708, 5.5306372769517395e-12)
(0.809, 6.8031668312019605e-12)
(0.91, 8.28759068975312e-12)
(1.011, 8.933631471519332e-12)
(1.112, 1.029733347526377e-11)
(1.213, 1.2267784440071628e-11)
(1.314, 1.4502625483842318e-11)
(1.415, 1.6872379588782955e-11)
(1.516, 1.9009891617152653e-11)
(1.617, 2.2461122162421748e-11)
(1.718, 2.598652418777247e-11)
(1.819, 3.0782160804845045e-11)
(1.92, 3.6364848880613376e-11)
(2.021, 4.28292097358122e-11)
(2.122, 4.829213057541297e-11)
(2.223, 5.437913909988626e-11)
(2.324, 6.275697567586051e-11)
(2.425, 7.14568411250596e-11)
(2.526, 8.296502959729566e-11)
(2.627, 9.605742112210684e-11)
(2.728, 1.0766337633612311e-10)
(2.829, 1.2244131171280514e-10)
(2.93, 1.3983981358096053e-10)
(3.031, 1.6168055192876092e-10)
(3.132, 1.7911288373117584e-10)
(3.233, 2.0524250235715737e-10)
(3.334, 2.3270841652006048e-10)
(3.435, 2.626341406022612e-10)
(3.536, 2.955295242549448e-10)
(3.637, 3.333102560994136e-10)
(3.738, 3.749520464769602e-10)
(3.839, 4.19122519412996e-10)
(3.94, 5.182727994187754e-10)
(4.041, 5.278166303044596e-10)
(4.142, 5.874777987080292e-10)
(4.243, 6.490018331789018e-10)
(4.344, 7.084980299453111e-10)
(4.445, 7.60976473921691e-10)
(4.546, 7.961017203470402e-10)
(4.647, 8.040832062212077e-10)
(4.748, 7.697244031553686e-10)
(4.849, 6.829814691873833e-10)
(4.95, 5.422902802120226e-10)
(5.051, 3.605057395957002e-10)
(5.152, 1.7228189925807747e-10)
(5.253, 1.5914699442919372e-10)
(5.354, 3.440404077388551e-10)
(5.455, 5.395121635625218e-10)
(5.556, 7.225510062874662e-10)
(5.657, 8.886750647757045e-10)
(5.758, 1.0350470550027819e-9)
(5.859, 1.1548529870266264e-9)
(5.96, 1.2354470789161375e-9)
(6.061, 1.257942283713418e-9)
(6.162, 1.2008935734914683e-9)
(6.263, 1.0505602523311236e-9)
(6.364, 8.082846296645766e-10)
(6.465, 5.02787849602474e-10)
(6.566, 2.2963178333436193e-10)
(6.667, 3.4926934472770807e-10)
(6.768, 6.638245046001087e-10)
(6.869, 9.696212787528887e-10)
(6.97, 1.2395194055481672e-9)
(7.071, 1.4614965282783135e-9)
(7.172, 1.619673312911773e-9)
(7.273, 1.689743681250532e-9)
(7.374, 1.6429849997166803e-9)
(7.475, 1.4565921298295239e-9)
(7.576, 1.1323758652690532e-9)
(7.677, 7.102112643066718e-10)
(7.778, 3.3059455248131493e-10)
(7.879, 4.975382130547896e-10)
(7.98, 9.34499337574408e-10)
(8.081, 1.3504350587846596e-9)
(8.182, 1.7003603776040916e-9)
(8.283, 1.9592655975336397e-9)
(8.384, 2.096361134975009e-9)
(8.485, 2.075201402537907e-9)
(8.586, 1.8621994270291374e-9)
(8.687, 1.4598487731242303e-9)
(8.788, 9.201229759214369e-10)
(8.889, 4.3400745573602274e-10)
(8.99, 6.566789793710642e-10)
(9.091, 1.218463782928552e-9)
(9.192, 1.7410823461674547e-9)
(9.293, 2.1589518456196597e-9)
(9.394, 2.4312499911642166e-9)
(9.495, 2.5121552438241805e-9)
(9.596, 2.3576886269377214e-9)
(9.697, 1.950174416915341e-9)
(9.798, 1.332706280686609e-9)
(9.899, 6.658332641302624e-10)
(10.0, 6.524101886852936e-10)
};
\addlegendentry{$n=8$}
\addplot+ [color = {rgb,1:red,0.24222430;green,0.64327509;blue,0.30444865},
draw opacity = 1.0,
line width = 2,
dotted,mark = none,
mark size = 2.0,
mark options = {
    color = {rgb,1:red,0.00000000;green,0.00000000;blue,0.00000000}, draw opacity = 1.0,
    fill = {rgb,1:red,0.24222430;green,0.64327509;blue,0.30444865}, fill opacity = 1.0,
    line width = 1,
    rotate = 0,
    solid
}]coordinates {
(0.001, 7.323183227975677e-13)
(0.102, 4.16919360759548e-13)
(0.203, 4.978105843268875e-13)
(0.304, 4.747178638287244e-13)
(0.405, 7.424239261543758e-13)
(0.506, 2.1975050875288968e-12)
(0.607, 2.465702733654596e-12)
(0.708, 1.3595401145744651e-12)
(0.809, 8.367620724301109e-13)
(0.91, 1.403836775885266e-12)
(1.011, 2.5505252518828674e-13)
(1.112, 1.4230126092053474e-12)
(1.213, 1.6025733213556102e-12)
(1.314, 1.4450481546406684e-12)
(1.415, 2.162130309141908e-12)
(1.516, 2.7332040363040066e-13)
(1.617, 3.7438330047441925e-12)
(1.718, 4.2327144485048866e-12)
(1.819, 1.8735563380713874e-12)
(1.92, 4.262133090941991e-12)
(2.021, 3.681798254829039e-12)
(2.122, 4.6952829517243475e-12)
(2.223, 6.939015022853074e-12)
(2.324, 6.55066183645068e-12)
(2.425, 2.768992834837457e-12)
(2.526, 3.4724318605600016e-12)
(2.627, 8.44151852653763e-12)
(2.728, 5.8595802270984115e-12)
(2.829, 8.14886489999315e-12)
(2.93, 6.0684936940811974e-12)
(3.031, 5.5356125997023636e-12)
(3.132, 5.992560461382177e-12)
(3.233, 4.629508640772742e-12)
(3.334, 8.388208196891826e-12)
(3.435, 6.893582493102803e-12)
(3.536, 7.545017423542205e-12)
(3.637, 4.781685905314656e-12)
(3.738, 2.053974119744404e-12)
(3.839, 4.400061556903154e-12)
(3.94, 1.952925139406252e-11)
(4.041, 8.341958797119421e-12)
(4.142, 6.4273459420018454e-12)
(4.243, 1.681208057488153e-12)
(4.344, 7.651769338718426e-12)
(4.445, 6.827591249030814e-12)
(4.546, 1.3250185444196262e-11)
(4.647, 1.5161590515003156e-11)
(4.748, 3.5247367444262037e-12)
(4.849, 2.7186000726062934e-12)
(4.95, 5.0578208370614214e-12)
(5.051, 2.6722739169711177e-12)
(5.152, 2.8510645644037998e-12)
(5.253, 1.2282598314507742e-11)
(5.354, 3.5133661971409773e-12)
(5.455, 7.1610981118469574e-12)
(5.556, 3.825537968158385e-13)
(5.657, 5.407318304073352e-13)
(5.758, 5.617843036324882e-13)
(5.859, 9.501136081905063e-13)
(5.96, 4.576227260335482e-13)
(6.061, 1.8836432174268676e-13)
(6.162, 2.3094645808318806e-13)
(6.263, 3.3083316540529466e-13)
(6.364, 1.9149771687418007e-12)
(6.465, 2.9327875092953217e-13)
(6.566, 4.447766829364594e-13)
(6.667, 2.412035089080072e-13)
(6.768, 9.183533508593275e-13)
(6.869, 2.1516251112711442e-13)
(6.97, 1.705778102873612e-13)
(7.071, 9.637912515991504e-14)
(7.172, 1.0975187749837043e-13)
(7.273, 1.8354511794510012e-13)
(7.374, 2.690499932903207e-13)
(7.475, 2.9709029954935526e-13)
(7.576, 3.916049340719504e-13)
(7.677, 5.583488199384895e-13)
(7.778, 4.709502995132075e-13)
(7.879, 4.095759601252955e-13)
(7.98, 3.477576162367677e-13)
(8.081, 3.272046314675545e-13)
(8.182, 2.404812010791648e-13)
(8.283, 8.35004523363828e-14)
(8.384, 1.9294661258192737e-13)
(8.485, 1.8197906746043917e-13)
(8.586, 4.022527890065014e-13)
(8.687, 4.893445261366854e-13)
(8.788, 5.814634107483475e-13)
(8.889, 5.766486125922745e-13)
(8.99, 3.5577667283449455e-13)
(9.091, 5.323211014939645e-13)
(9.192, 4.1911742648190393e-13)
(9.293, 1.7812174764597176e-13)
(9.394, 1.0516678891962581e-13)
(9.495, 2.2481267563807165e-13)
(9.596, 4.1626425853932425e-13)
(9.697, 5.358311314556902e-13)
(9.798, 6.804687674832003e-13)
(9.899, 7.488921795123914e-13)
(10.0, 7.12494221904191e-13)
};
\addlegendentry{$n=10$}
\end{axis}

\end{tikzpicture}

%% file: error_odd_n_v2.tikz
\begin{tikzpicture}[]
\begin{axis}[height = \figureheight, ylabel = {}, xmin = {-0.29897}, xmax = {10.29997}, ymax = {2.8183829312644537}, ymode = {log}, xlabel = {$c$}, unbounded coords=jump,scaled x ticks = false,xlabel style = {font = {\fontsize{11 pt}{14.3 pt}\selectfont}, color = {rgb,1:red,0.00000000;green,0.00000000;blue,0.00000000}, draw opacity = 1.0, rotate = 0.0},xmajorgrids = true,xtick = {0.0,2.5,5.0,7.5,10.0},xticklabels = {$0.0$,$2.5$,$5.0$,$7.5$,$10.0$},xtick align = inside,xticklabel style = {font = {\fontsize{8 pt}{10.4 pt}\selectfont}, color = {rgb,1:red,0.00000000;green,0.00000000;blue,0.00000000}, draw opacity = 1.0, rotate = 0.0},x grid style = {color = {rgb,1:red,0.00000000;green,0.00000000;blue,0.00000000},
draw opacity = 0.1,
line width = 0.5,
solid},axis x line* = left,x axis line style = {color = {rgb,1:red,0.00000000;green,0.00000000;blue,0.00000000},
draw opacity = 1.0,
line width = 1,
solid},scaled y ticks = false,ylabel style = {font = {\fontsize{11 pt}{14.3 pt}\selectfont}, color = {rgb,1:red,0.00000000;green,0.00000000;blue,0.00000000}, draw opacity = 1.0, rotate = 0.0},log basis y=10,ymajorgrids = true,ytick = {1.0e-15,1.0e-10,1.0e-5,1.0},yticklabels = {$10^{-15}$,$10^{-10}$,$10^{-5}$,$10^{0}$},ytick align = inside,yticklabel style = {font = {\fontsize{8 pt}{10.4 pt}\selectfont}, color = {rgb,1:red,0.00000000;green,0.00000000;blue,0.00000000}, draw opacity = 1.0, rotate = 0.0},y grid style = {color = {rgb,1:red,0.00000000;green,0.00000000;blue,0.00000000},
draw opacity = 0.1,
line width = 0.5,
solid},axis y line* = left,y axis line style = {color = {rgb,1:red,0.00000000;green,0.00000000;blue,0.00000000},
draw opacity = 1.0,
line width = 1,
solid},    xshift = 0.0mm,
    yshift = 0.0mm,
    axis background/.style={fill={rgb,1:red,1.00000000;green,1.00000000;blue,1.00000000}}
,legend style = {color = {rgb,1:red,0.00000000;green,0.00000000;blue,0.00000000},
draw opacity = 1.0,
line width = 1,
solid,fill = {rgb,1:red,1.00000000;green,1.00000000;blue,1.00000000},font = {\fontsize{8 pt}{10.4 pt}\selectfont}},colorbar style={title=}, ymin = {3.5481338923357605e-16}, width = \figurewidth]\addplot+ [color = {rgb,1:red,0.00000000;green,0.60560316;blue,0.97868012},
draw opacity = 1.0,
line width = 2,
solid,mark = none,
mark size = 2.0,
mark options = {
    color = {rgb,1:red,0.00000000;green,0.00000000;blue,0.00000000}, draw opacity = 1.0,
    fill = {rgb,1:red,0.00000000;green,0.60560316;blue,0.97868012}, fill opacity = 1.0,
    line width = 1,
    rotate = 0,
    solid
},forget plot]coordinates {
(0.001, 1.8176564614007656e-10)
(0.05362631578947368, 1.984054214893806e-10)
(0.10625263157894736, 2.169870100223263e-10)
(0.15887894736842106, 2.369548904725589e-10)
(0.21150526315789472, 2.5824640141304064e-10)
(0.26413157894736844, 2.819711389038755e-10)
(0.3167578947368421, 3.0726134395500136e-10)
(0.3693842105263158, 3.349997637719269e-10)
(0.42201052631578945, 3.6459179203020903e-10)
(0.47463684210526313, 3.9683034336850344e-10)
(0.5272631578947369, 4.314868538172889e-10)
(0.5798894736842105, 4.691464925403999e-10)
(0.6325157894736843, 5.096796068552024e-10)
(0.6851421052631579, 5.53904366823945e-10)
(0.7377684210526316, 6.011124709023263e-10)
(0.7903947368421053, 6.518314426390929e-10)
(0.8430210526315789, 7.069783806813767e-10)
(0.8956473684210526, 7.659668263375495e-10)
(0.9482736842105263, 8.293206158613059e-10)
(1.0009, 8.984415697014419e-10)
(1.0535263157894736, 9.719525022817168e-10)
(1.1061526315789474, 1.0507116056511696e-9)
(1.1587789473684211, 1.1354799046397723e-9)
(1.2114052631578947, 1.2266628015591458e-9)
(1.2640315789473684, 1.3245300007742905e-9)
(1.3166578947368421, 1.4288428790269744e-9)
(1.3692842105263159, 1.5409291962371313e-9)
(1.4219105263157894, 1.6618017663759764e-9)
(1.4745368421052631, 1.7897331174699103e-9)
(1.527163157894737, 1.927271332009854e-9)
(1.5797894736842106, 2.0742094892656295e-9)
(1.6324157894736842, 2.231604003149529e-9)
(1.685042105263158, 2.3989925177663736e-9)
(1.7376684210526316, 2.577829954405169e-9)
(1.7902947368421052, 2.7680453454746028e-9)
(1.842921052631579, 2.9718730712250626e-9)
(1.8955473684210526, 3.189206237337981e-9)
(1.9481736842105264, 3.4191290443197558e-9)
(2.0008, 3.664822217680259e-9)
(2.0534263157894737, 3.92645424903574e-9)
(2.1060526315789474, 4.204426941527409e-9)
(2.158678947368421, 4.500020968619853e-9)
(2.211305263157895, 4.813328342136689e-9)
(2.2639315789473686, 5.147304292358541e-9)
(2.316557894736842, 5.500260515234132e-9)
(2.3691842105263157, 5.875604321563685e-9)
(2.4218105263157894, 6.274608061064003e-9)
(2.474436842105263, 6.697336819062855e-9)
(2.527063157894737, 7.144889058884221e-9)
(2.5796894736842106, 7.619245240473826e-9)
(2.6323157894736844, 8.121644912591649e-9)
(2.6849421052631577, 8.653738009401679e-9)
(2.7375684210526314, 9.217796884746551e-9)
(2.790194736842105, 9.813608404666268e-9)
(2.842821052631579, 1.0445348407796284e-8)
(2.8954473684210527, 1.1113806172733628e-8)
(2.9480736842105264, 1.1821257824660519e-8)
(3.0007, 1.2570025818945718e-8)
(3.053326315789474, 1.3362921901318609e-8)
(3.105952631578947, 1.4202068005101205e-8)
(3.158578947368421, 1.5091828253496796e-8)
(3.2112052631578947, 1.6035359972089083e-8)
(3.2638315789473684, 1.7034672604708396e-8)
(3.316457894736842, 1.8097388456483814e-8)
(3.369084210526316, 1.9226175931880587e-8)
(3.4217105263157896, 2.042865273948199e-8)
(3.474336842105263, 2.1711123124728587e-8)
(3.5269631578947367, 2.308195483241801e-8)
(3.5795894736842104, 2.4552249432844295e-8)
(3.632215789473684, 2.613491549434382e-8)
(3.684842105263158, 2.7844667734529e-8)
(3.7374684210526317, 2.970347867984718e-8)
(3.7900947368421054, 3.173706594750094e-8)
(3.842721052631579, 3.3979755089125503e-8)
(3.8953473684210524, 3.6477400387338325e-8)
(3.947973684210526, 3.9291464812381054e-8)
(4.0006, 4.251149373639358e-8)
(4.053226315789474, 4.625744886540946e-8)
(4.105852631578947, 5.071105117730106e-8)
(4.158478947368421, 5.6148023535337526e-8)
(4.2111052631578945, 6.300858357863252e-8)
(4.263731578947368, 7.204548589841744e-8)
(4.316357894736842, 8.466232102073086e-8)
(4.368984210526316, 1.0384412411974113e-7)
(4.421610526315789, 1.3738608798807733e-7)
(4.474236842105263, 2.1453251234851285e-7)
(4.526863157894737, 5.877538624545444e-7)
(4.579489473684211, 0.0015785127781189666)
(4.632115789473684, 483.5042378264375)
(4.684742105263158, 1.2972363721156196e-6)
(4.737368421052632, 3.256832569467673e-7)
(4.789994736842106, 1.8625393060038315e-7)
(4.842621052631579, 1.3370293412403552e-7)
(4.895247368421052, 1.06859298058949e-7)
(4.947873684210526, 9.111790867080674e-8)
(5.0005, 8.126065608935166e-8)
(5.0531263157894735, 7.496502818367746e-8)
(5.105752631578947, 7.104527263524638e-8)
(5.158378947368421, 6.884100089515656e-8)
(5.211005263157895, 6.79660780466182e-8)
(5.263631578947368, 6.819428827841843e-8)
(5.316257894736842, 6.940470901450784e-8)
(5.368884210526316, 7.155585479733789e-8)
(5.42151052631579, 7.467836626822198e-8)
(5.474136842105263, 7.88813355558642e-8)
(5.526763157894737, 8.437724743300479e-8)
(5.579389473684211, 9.152135652629684e-8)
(5.632015789473685, 1.0092607184343264e-7)
(5.6846421052631575, 1.1364349906521639e-7)
(5.737268421052631, 1.316331614852755e-7)
(5.789894736842105, 1.5896843321843077e-7)
(5.842521052631579, 2.0577688232127846e-7)
(5.8951473684210525, 3.0636074901205414e-7)
(5.947773684210526, 6.841693922972718e-7)
(6.0004, 3.4144642430258104e-5)
(6.053026315789474, 4.207047517488459e29)
(6.105652631578947, 4.331995150845003e-6)
(6.158278947368421, 6.02683056581608e-7)
(6.210905263157895, 3.149985170822305e-7)
(6.263531578947369, 2.1955410940544834e-7)
(6.316157894736842, 1.7338972219576934e-7)
(6.368784210526316, 1.471926167342869e-7)
(6.42141052631579, 1.3123367941040953e-7)
(6.474036842105263, 1.2137877413972499e-7)
(6.5266631578947365, 1.1559893996977062e-7)
(6.57928947368421, 1.12816982773683e-7)
(6.631915789473684, 1.1245699531650153e-7)
(6.684542105263158, 1.1425248537881628e-7)
(6.7371684210526315, 1.1817139267763137e-7)
(6.789794736842105, 1.2440442074906585e-7)
(6.842421052631579, 1.3343385447510054e-7)
(6.895047368421053, 1.4616552272971108e-7)
(6.9476736842105264, 1.6429019543503013e-7)
(7.0003, 1.9111645304920277e-7)
(7.052926315789474, 2.3393194480752274e-7)
(7.105552631578948, 3.124308070416414e-7)
(7.158178947368421, 5.040796416802432e-7)
(7.210805263157895, 1.5649637785974256e-6)
(7.263431578947368, 0.07239612509415647)
(7.316057894736842, 1.9452963009145028)
(7.3686842105263155, 2.1426741267329564e-6)
(7.421310526315789, 6.177430566646181e-7)
(7.473936842105263, 3.6706911174893115e-7)
(7.526563157894737, 2.700005772854736e-7)
(7.5791894736842105, 2.2031485051682934e-7)
(7.631815789473684, 1.916550057102248e-7)
(7.684442105263158, 1.7446290256566388e-7)
(7.737068421052632, 1.6449776280264864e-7)
(7.7896947368421054, 1.5967047093527616e-7)
(7.842321052631579, 1.589448414165956e-7)
(7.894947368421053, 1.619133804593179e-7)
(7.947573684210527, 1.686629031279938e-7)
(8.0002, 1.7977013594887504e-7)
(8.052826315789474, 1.9659821731769134e-7)
(8.105452631578947, 2.2181927398508332e-7)
(8.158078947368422, 2.6101429330981886e-7)
(8.210705263157895, 3.2747164645692895e-7)
(8.26333157894737, 4.6185516660183085e-7)
(8.315957894736842, 8.698591837270987e-7)
(8.368584210526317, 6.855936506231768e-6)
(8.42121052631579, 1.9872441394301522e37)
(8.473836842105262, 5.406337312434476e-5)
(8.526463157894737, 1.3514026549904235e-6)
(8.57908947368421, 5.973835123258416e-7)
(8.631715789473684, 3.96941068938261e-7)
(8.684342105263157, 3.080660047873784e-7)
(8.736968421052632, 2.602495388383839e-7)
(8.789594736842105, 2.3255055390630968e-7)
(8.84222105263158, 2.1668650532406823e-7)
(8.894847368421052, 2.088705656840903e-7)
(8.947473684210527, 2.0733012612045793e-7)
(9.0001, 2.114313688513232e-7)
(9.052726315789474, 2.214158612678189e-7)
(9.105352631578947, 2.3851042542717194e-7)
(9.157978947368422, 2.654958428768403e-7)
(9.210605263157895, 3.083188734654588e-7)
(9.263231578947368, 3.809914220888437e-7)
(9.315857894736842, 5.247463803989535e-7)
(9.368484210526315, 9.29368481987882e-7)
(9.42111052631579, 4.776770452519203e-6)
(9.473736842105263, 1.5027187314905159e10)
(9.526363157894737, 0.00045196022295947)
(9.57898947368421, 1.9396740179058326e-6)
(9.631615789473685, 7.651839968974115e-7)
(9.684242105263158, 4.936225490780615e-7)
(9.736868421052632, 3.7890114683803503e-7)
(9.789494736842105, 3.1885711662945274e-7)
(9.84212105263158, 2.8493771529855596e-7)
(9.894747368421053, 2.6626897762426006e-7)
(9.947373684210527, 2.5807116396379894e-7)
(10.0, 2.582922633932652e-7)
};
\addplot+ [color = {rgb,1:red,0.88887350;green,0.43564919;blue,0.27812294},
draw opacity = 1.0,
line width = 2,
dashed,mark = none,
mark size = 2.0,
mark options = {
    color = {rgb,1:red,0.00000000;green,0.00000000;blue,0.00000000}, draw opacity = 1.0,
    fill = {rgb,1:red,0.88887350;green,0.43564919;blue,0.27812294}, fill opacity = 1.0,
    line width = 1,
    rotate = 0,
    solid
},forget plot]coordinates {
(0.001, 3.1329958703112415e-13)
(0.05362631578947368, 4.682124174336993e-13)
(0.10625263157894736, 3.3760186317137587e-13)
(0.15887894736842106, 2.936135868686054e-13)
(0.21150526315789472, 5.58970203682593e-13)
(0.26413157894736844, 3.47835265565157e-13)
(0.3167578947368421, 8.863096136174362e-13)
(0.3693842105263158, 8.892015739688747e-13)
(0.42201052631578945, 4.145472870514025e-13)
(0.47463684210526313, 2.0902076626278454e-13)
(0.5272631578947369, 7.999528781526885e-13)
(0.5798894736842105, 2.2303405160793917e-13)
(0.6325157894736843, 5.385367495476128e-13)
(0.6851421052631579, 2.49488852700803e-13)
(0.7377684210526316, 5.975873611525338e-13)
(0.7903947368421053, 9.331049987556455e-13)
(0.8430210526315789, 8.275914325472703e-13)
(0.8956473684210526, 1.4139345369725486e-12)
(0.9482736842105263, 8.396806981536469e-13)
(1.0009, 1.5409810195171647e-12)
(1.0535263157894736, 4.4740435481654455e-13)
(1.1061526315789474, 1.0573813978837162e-12)
(1.1587789473684211, 8.856594430944776e-13)
(1.2114052631578947, 4.3658836870964095e-13)
(1.2640315789473684, 1.0049106908958157e-12)
(1.3166578947368421, 5.153656715324457e-13)
(1.3692842105263159, 6.813646289140013e-13)
(1.4219105263157894, 4.24847929444587e-13)
(1.4745368421052631, 2.292793195800549e-12)
(1.527163157894737, 1.8924751736313207e-12)
(1.5797894736842106, 1.6085068119315803e-12)
(1.6324157894736842, 1.1028212521664376e-12)
(1.685042105263158, 9.565120782556764e-13)
(1.7376684210526316, 2.7817136004529104e-12)
(1.7902947368421052, 2.099596913236181e-12)
(1.842921052631579, 8.752451998817923e-13)
(1.8955473684210526, 1.1118823729209844e-12)
(1.9481736842105264, 1.5461633427235806e-12)
(2.0008, 9.800356896877906e-13)
(2.0534263157894737, 1.1030329417483567e-12)
(2.1060526315789474, 2.1944681398976676e-12)
(2.158678947368421, 1.543322176607367e-12)
(2.211305263157895, 2.7533357632940405e-12)
(2.2639315789473686, 8.048749398036781e-13)
(2.316557894736842, 4.380138038702409e-12)
(2.3691842105263157, 3.478503399649074e-12)
(2.4218105263157894, 2.575942284902367e-12)
(2.474436842105263, 5.031217168724193e-12)
(2.527063157894737, 3.126637157858374e-12)
(2.5796894736842106, 4.665373741310132e-12)
(2.6323157894736844, 3.5059475966312244e-12)
(2.6849421052631577, 9.286407181496718e-13)
(2.7375684210526314, 1.1496144127988868e-12)
(2.790194736842105, 1.53434387012024e-12)
(2.842821052631579, 3.488270903560484e-12)
(2.8954473684210527, 1.672077112539106e-12)
(2.9480736842105264, 2.9975599304281583e-12)
(3.0007, 2.0710770864256354e-12)
(3.053326315789474, 4.3239139080054993e-13)
(3.105952631578947, 3.956209529955671e-12)
(3.158578947368421, 2.7813567018596913e-12)
(3.2112052631578947, 4.926540027396396e-12)
(3.2638315789473684, 3.2754973246993958e-12)
(3.316457894736842, 1.844536775328628e-12)
(3.369084210526316, 4.39398933545044e-12)
(3.4217105263157896, 1.2714360650290357e-12)
(3.474336842105263, 7.226743565227488e-13)
(3.5269631578947367, 4.8099116837147e-12)
(3.5795894736842104, 3.5270875845568024e-12)
(3.632215789473684, 5.234474026373012e-12)
(3.684842105263158, 7.476925113206603e-12)
(3.7374684210526317, 5.846352723906399e-12)
(3.7900947368421054, 6.7383835986445384e-12)
(3.842721052631579, 4.442642356813393e-12)
(3.8953473684210524, 2.7327542918005097e-12)
(3.947973684210526, 2.453018920439951e-11)
(4.0006, 4.613730657218843e-12)
(4.053226315789474, 1.0556885348572416e-11)
(4.105852631578947, 4.611442307032779e-12)
(4.158478947368421, 7.846338669667189e-12)
(4.2111052631578945, 6.1994916214670114e-12)
(4.263731578947368, 9.420079172176616e-12)
(4.316357894736842, 1.0117617687754395e-11)
(4.368984210526316, 9.184722618702765e-12)
(4.421610526315789, 1.0055198870077828e-11)
(4.474236842105263, 9.299978829504966e-12)
(4.526863157894737, 1.1855231170075972e-11)
(4.579489473684211, 1.5837268140868677e-11)
(4.632115789473684, 2.630788456596777e-11)
(4.684742105263158, 1.2230916842772027e-11)
(4.737368421052632, 1.7384475594608858e-11)
(4.789994736842106, 1.930935982511338e-11)
(4.842621052631579, 1.9656797626527446e-11)
(4.895247368421052, 2.2139966217188972e-11)
(4.947873684210526, 2.7183980249502133e-11)
(5.0005, 3.170907245270875e-11)
(5.0531263157894735, 4.3053272544205994e-11)
(5.105752631578947, 6.644680812997513e-11)
(5.158378947368421, 1.779750320044446e-10)
(5.211005263157895, 0.0005158940859063286)
(5.263631578947368, 6.449954984205505e-8)
(5.316257894736842, 1.8111620388682313e-10)
(5.368884210526316, 7.632762955122304e-11)
(5.42151052631579, 5.039114983545134e-11)
(5.474136842105263, 3.8347477637289476e-11)
(5.526763157894737, 3.799234641975345e-11)
(5.579389473684211, 2.7492704957553535e-11)
(5.632015789473685, 2.5073334887969097e-11)
(5.6846421052631575, 2.3581292110995428e-11)
(5.737268421052631, 2.2660985650341917e-11)
(5.789894736842105, 2.275008856843808e-11)
(5.842521052631579, 2.284453468586727e-11)
(5.8951473684210525, 2.3103547832926143e-11)
(5.947773684210526, 2.3913323726321495e-11)
(6.0004, 2.5032941146333218e-11)
(6.053026315789474, 2.6183957609344198e-11)
(6.105652631578947, 2.8397863368098996e-11)
(6.158278947368421, 3.100727737501325e-11)
(6.210905263157895, 3.422821752485595e-11)
(6.263531578947369, 3.869594731872173e-11)
(6.316157894736842, 4.484395974725684e-11)
(6.368784210526316, 5.482472132477031e-11)
(6.42141052631579, 6.979298824941675e-11)
(6.474036842105263, 1.0285142725875204e-10)
(6.5266631578947365, 2.293525017656023e-10)
(6.57928947368421, 9.413261299573863e-8)
(6.631915789473684, 0.0008953616685627458)
(6.684542105263158, 4.1052346573789647e-10)
(6.7371684210526315, 1.3885442868781488e-10)
(6.789794736842105, 9.646955990438e-11)
(6.842421052631579, 6.624147861740519e-11)
(6.895047368421053, 5.404213561808183e-11)
(6.9476736842105264, 4.763980491117224e-11)
(7.0003, 4.34827306376936e-11)
(7.052926315789474, 4.1061715386813035e-11)
(7.105552631578948, 3.987564058316653e-11)
(7.158178947368421, 3.949628368924604e-11)
(7.210805263157895, 4.023280974225984e-11)
(7.263431578947368, 4.159927199999523e-11)
(7.316057894736842, 4.356640096618306e-11)
(7.3686842105263155, 4.7014045537667615e-11)
(7.421310526315789, 5.115720654807221e-11)
(7.473936842105263, 5.7178712392018434e-11)
(7.526563157894737, 6.591557579974846e-11)
(7.5791894736842105, 7.931307998979095e-11)
(7.631815789473684, 1.0236110843115766e-10)
(7.684442105263158, 1.5222699095782447e-10)
(7.737068421052632, 3.507971751754456e-10)
(7.7896947368421054, 3.109531284563077e-7)
(7.842321052631579, 0.00016149039327608916)
(7.894947368421053, 5.467617844151443e-10)
(7.947573684210527, 1.9895478419842813e-10)
(8.0002, 1.2392003001244866e-10)
(8.052826315789474, 9.371111084947345e-11)
(8.105452631578947, 7.776756752209521e-11)
(8.158078947368422, 6.852913269526059e-11)
(8.210705263157895, 6.312684274756921e-11)
(8.26333157894737, 6.023598065083246e-11)
(8.315957894736842, 5.917921359331087e-11)
(8.368584210526317, 5.952908106545403e-11)
(8.42121052631579, 6.114955647975056e-11)
(8.473836842105262, 6.454949989084164e-11)
(8.526463157894737, 6.953181321977926e-11)
(8.57908947368421, 7.687270965111963e-11)
(8.631715789473684, 8.780256119786227e-11)
(8.684342105263157, 1.048666967203346e-10)
(8.736968421052632, 1.3440164355654042e-10)
(8.789594736842105, 1.9761176020169384e-10)
(8.84222105263158, 4.3748756252512365e-10)
(8.894847368421052, 1.095093386188373e-7)
(8.947473684210527, 0.01164851583798756)
(9.0001, 7.860557407846438e-10)
(9.052726315789474, 2.651060840285165e-10)
(9.105352631578947, 1.6539883296552953e-10)
(9.157978947368422, 1.2481674044164509e-10)
(9.210605263157895, 1.0375431579695928e-10)
(9.263231578947368, 9.182228337648661e-11)
(9.315857894736842, 8.50351080963352e-11)
(9.368484210526315, 8.167524010942564e-11)
(9.42111052631579, 8.073966562424403e-11)
(9.473736842105263, 8.229907503043398e-11)
(9.526363157894737, 8.595750053543923e-11)
(9.57898947368421, 9.225941931635638e-11)
(9.631615789473685, 1.0212249601204806e-10)
(9.684242105263158, 1.1728362947092253e-10)
(9.736868421052632, 1.4189317770173987e-10)
(9.789494736842105, 1.8723565460508712e-10)
(9.84212105263158, 2.966646671788714e-10)
(9.894747368421053, 9.248481174706578e-10)
(9.947373684210527, 8.45107687371805)
(10.0, 5.2056838355532486e-8)
};
\addplot+ [color = {rgb,1:red,0.24222430;green,0.64327509;blue,0.30444865},
draw opacity = 1.0,
line width = 2,
dotted,mark = none,
mark size = 2.0,
mark options = {
    color = {rgb,1:red,0.00000000;green,0.00000000;blue,0.00000000}, draw opacity = 1.0,
    fill = {rgb,1:red,0.24222430;green,0.64327509;blue,0.30444865}, fill opacity = 1.0,
    line width = 1,
    rotate = 0,
    solid
},forget plot]coordinates {
(0.001, 4.575710083998578e-13)
(0.05362631578947368, 1.1151985533028364e-12)
(0.10625263157894736, 1.488376183410153e-12)
(0.15887894736842106, 1.5299460606656596e-12)
(0.21150526315789472, 3.411590890368299e-13)
(0.26413157894736844, 2.039603708668927e-12)
(0.3167578947368421, 2.3670223170027442e-12)
(0.3693842105263158, 1.830759706624844e-12)
(0.42201052631578945, 1.3774352729074967e-12)
(0.47463684210526313, 1.070646824678699e-12)
(0.5272631578947369, 1.859000746141257e-12)
(0.5798894736842105, 3.0681159018262974e-12)
(0.6325157894736843, 1.2152363133518424e-12)
(0.6851421052631579, 1.3321793928098066e-12)
(0.7377684210526316, 2.3823860097873783e-12)
(0.7903947368421053, 4.884636243893802e-13)
(0.8430210526315789, 9.064122495519464e-13)
(0.8956473684210526, 3.1879386895206543e-12)
(0.9482736842105263, 7.697019733608681e-12)
(1.0009, 2.4556584991671034e-12)
(1.0535263157894736, 2.897855141396496e-12)
(1.1061526315789474, 3.4162400049820272e-12)
(1.1587789473684211, 1.1160539250597653e-12)
(1.2114052631578947, 2.911394753682989e-12)
(1.2640315789473684, 1.5243959679985609e-12)
(1.3166578947368421, 3.2360706371635676e-12)
(1.3692842105263159, 9.411042516496818e-13)
(1.4219105263157894, 5.27852103866927e-12)
(1.4745368421052631, 3.832852577436645e-12)
(1.527163157894737, 4.964703087462925e-12)
(1.5797894736842106, 4.513123628586004e-12)
(1.6324157894736842, 2.659616022213741e-12)
(1.685042105263158, 5.442941361453287e-12)
(1.7376684210526316, 7.91130209822865e-12)
(1.7902947368421052, 3.1845426680782084e-12)
(1.842921052631579, 4.613288498518925e-12)
(1.8955473684210526, 4.98309306550753e-12)
(1.9481736842105264, 7.536909650072054e-12)
(2.0008, 6.710510388930298e-12)
(2.0534263157894737, 3.428900020401155e-12)
(2.1060526315789474, 1.1197304930807776e-11)
(2.158678947368421, 6.4508886323566e-12)
(2.211305263157895, 9.66542326792596e-12)
(2.2639315789473686, 5.227254716229443e-12)
(2.316557894736842, 7.51865369383222e-12)
(2.3691842105263157, 2.0455225278911828e-11)
(2.4218105263157894, 4.233047583600902e-12)
(2.474436842105263, 8.50930050757032e-12)
(2.527063157894737, 1.0719097587988333e-11)
(2.5796894736842106, 1.2172678283981404e-11)
(2.6323157894736844, 3.207733801136301e-12)
(2.6849421052631577, 1.2532783538098703e-11)
(2.7375684210526314, 1.8883988581012782e-11)
(2.790194736842105, 2.231449008989043e-12)
(2.842821052631579, 3.103302502151495e-11)
(2.8954473684210527, 1.3307263441250401e-11)
(2.9480736842105264, 5.589518435699823e-12)
(3.0007, 7.337608191701943e-12)
(3.053326315789474, 2.6147351784905745e-11)
(3.105952631578947, 1.1740652289664309e-11)
(3.158578947368421, 8.068806099566553e-12)
(3.2112052631578947, 3.535856647233768e-11)
(3.2638315789473684, 1.3872688739797421e-11)
(3.316457894736842, 1.970958796188553e-11)
(3.369084210526316, 1.682723964322417e-11)
(3.4217105263157896, 1.5722445669198293e-11)
(3.474336842105263, 4.352755749828787e-12)
(3.5269631578947367, 1.9010824122354907e-11)
(3.5795894736842104, 8.021833195349379e-12)
(3.632215789473684, 2.9902772584164833e-12)
(3.684842105263158, 1.7044128834826923e-11)
(3.7374684210526317, 2.088562311122868e-11)
(3.7900947368421054, 1.8515113163860755e-11)
(3.842721052631579, 1.0875263332089504e-11)
(3.8953473684210524, 2.6657880891690915e-11)
(3.947973684210526, 1.149229589726837e-10)
(4.0006, 2.4776368706921522e-11)
(4.053226315789474, 2.456753193951483e-11)
(4.105852631578947, 1.4544314873947057e-11)
(4.158478947368421, 3.180067009787722e-11)
(4.2111052631578945, 1.4007753619339526e-11)
(4.263731578947368, 1.2296188788604094e-11)
(4.316357894736842, 2.039481622143071e-11)
(4.368984210526316, 1.6181405975494435e-11)
(4.421610526315789, 1.632884947891232e-11)
(4.474236842105263, 1.58298900752889e-11)
(4.526863157894737, 3.3893004592562995e-11)
(4.579489473684211, 6.804217003817193e-11)
(4.632115789473684, 3.763840704337976e-11)
(4.684742105263158, 2.3064241099364623e-11)
(4.737368421052632, 6.303800310627144e-12)
(4.789994736842106, 3.133057844963832e-11)
(4.842621052631579, 2.8739384832800477e-12)
(4.895247368421052, 9.155657589752305e-12)
(4.947873684210526, 7.713233512664595e-12)
(5.0005, 6.4088944810727856e-12)
(5.0531263157894735, 4.1288014071118395e-12)
(5.105752631578947, 1.2597936122953864e-11)
(5.158378947368421, 3.3576833037398764e-11)
(5.211005263157895, 6.582525800457031e-11)
(5.263631578947368, 3.2993308022963345e-11)
(5.316257894736842, 6.150855292488592e-12)
(5.368884210526316, 3.5161880756899396e-12)
(5.42151052631579, 9.520083028179671e-12)
(5.474136842105263, 4.5351050509482935e-12)
(5.526763157894737, 8.58910425960918e-12)
(5.579389473684211, 1.6422781579430367e-12)
(5.632015789473685, 2.155857681129653e-12)
(5.6846421052631575, 2.2342545332417024e-12)
(5.737268421052631, 1.980229615303609e-12)
(5.789894736842105, 6.257667824840009e-11)
(5.842521052631579, 1.2525588777772058e-11)
(5.8951473684210525, 1.7292186937781082e-12)
(5.947773684210526, 2.1613590906883304e-12)
(6.0004, 3.3177903920545497e-12)
(6.053026315789474, 1.025726894054974e-11)
(6.105652631578947, 1.5134621661332302e-12)
(6.158278947368421, 1.9921380970507075e-12)
(6.210905263157895, 2.4330101267912158e-14)
(6.263531578947369, 3.476279198546907e-13)
(6.316157894736842, 1.32404112917436e-12)
(6.368784210526316, 3.856770015918864e-11)
(6.42141052631579, 1.9756371151073416e-13)
(6.474036842105263, 6.581803990427812e-13)
(6.5266631578947365, 1.2805319296496137e-13)
(6.57928947368421, 3.4565128140448944e-13)
(6.631915789473684, 1.0473140716077984e-12)
(6.684542105263158, 5.798802414125045e-13)
(6.7371684210526315, 1.2822297525773895e-11)
(6.789794736842105, 1.0964640474173916e-11)
(6.842421052631579, 1.8688408444994365e-13)
(6.895047368421053, 1.4163855843930007e-11)
(6.9476736842105264, 1.091084916165344e-13)
(7.0003, 5.778153571798005e-14)
(7.052926315789474, 1.3063909933916816e-13)
(7.105552631578948, 1.595952838467181e-13)
(7.158178947368421, 4.039133371163102e13)
(7.210805263157895, 1.0427342371793007e-13)
(7.263431578947368, 1.1216397281637544e-13)
(7.316057894736842, 1.0578436403647866e-13)
(7.3686842105263155, 1.819509217766204e-13)
(7.421310526315789, 1.303883748182477e-13)
(7.473936842105263, 2.5919977886949132e-14)
(7.526563157894737, 4.1998872723274687e-14)
(7.5791894736842105, 4.1510465379655425e-14)
(7.631815789473684, 3.97217876996388e-14)
(7.684442105263158, 5.340311236009874e-14)
(7.737068421052632, 1.3468519976073855e-14)
(7.7896947368421054, 1.727850904364207e-14)
(7.842321052631579, 2.7067194078970166e-14)
(7.894947368421053, 5.167974519521708e-13)
(7.947573684210527, 1.0302155590916602e-11)
(8.0002, 2.520377439847492e-14)
(8.052826315789474, 3.722973626520346e-14)
(8.105452631578947, 4.0696754126873045e-14)
(8.158078947368422, 3.0364021265406716e-13)
(8.210705263157895, 2.1598403816247595e-14)
(8.26333157894737, 1.0258306048085994e-13)
(8.315957894736842, 3.0816459778922573e6)
(8.368584210526317, 1.8752144478275973e-12)
(8.42121052631579, 1.0860118334300308e-13)
(8.473836842105262, 3.1651657717144664e-13)
(8.526463157894737, 5.4268272980238763e-14)
(8.57908947368421, 7.03884899884471e-14)
(8.631715789473684, 7.853516019975181e-14)
(8.684342105263157, 1.9509622927395043e-14)
(8.736968421052632, 3.983149460152307e-14)
(8.789594736842105, 4.1921762670630697e-14)
(8.84222105263158, 8.579122260155907e-15)
(8.894847368421052, 6.129657850041726e-14)
(8.947473684210527, 9.77138183282905e-15)
(9.0001, 1.5798505838825758e-13)
(9.052726315789474, 2.1876978505531787e-13)
(9.105352631578947, 5.51109039742533e-13)
(9.157978947368422, 5.6291325070572595e-14)
(9.210605263157895, 2.7481608224845162e-14)
(9.263231578947368, 7.927769824805054e-14)
(9.315857894736842, 1.0307211917836712e-13)
(9.368484210526315, 6.919712670244279e-13)
(9.42111052631579, 1302.6601315372848)
(9.473736842105263, 2.7163138540126655e-13)
(9.526363157894737, 8.80068316402875e-13)
(9.57898947368421, 1.2111033388864091e-13)
(9.631615789473685, 3.0045123952516925e-14)
(9.684242105263158, 1.5744951274340434e-14)
(9.736868421052632, 7.701024542294895e-14)
(9.789494736842105, 1.656028642052358e-14)
(9.84212105263158, 1.7244582700448388e-14)
(9.894747368421053, 4.799616862662093e-14)
(9.947373684210527, 3.777855280410978e-14)
(10.0, 2.4079929452504017e-14)
};
\end{axis}

\end{tikzpicture}

%% file: relerror_chebyshev_n4_v2.tikz
\begin{tikzpicture}[]
\begin{axis}[height = \figureheight, ylabel = {}, xmin = {8.709635899560805}, xmax = {1148.1536214968828}, ymax = {1.0e-5}, ymode = {log}, xlabel = {k}, unbounded coords=jump,scaled x ticks = false,xlabel style = {font = {\fontsize{11 pt}{14.3 pt}\selectfont}, color = {rgb,1:red,0.00000000;green,0.00000000;blue,0.00000000}, draw opacity = 1.0, rotate = 0.0},log basis x=10,xmajorgrids = true,xtick = {10.0,31.622776601683793,100.0,316.22776601683796,1000.0},xticklabels = {$10^{1.0}$,$10^{1.5}$,$10^{2.0}$,$10^{2.5}$,$10^{3.0}$},xtick align = inside,xticklabel style = {font = {\fontsize{8 pt}{10.4 pt}\selectfont}, color = {rgb,1:red,0.00000000;green,0.00000000;blue,0.00000000}, draw opacity = 1.0, rotate = 0.0},x grid style = {color = {rgb,1:red,0.00000000;green,0.00000000;blue,0.00000000},
draw opacity = 0.1,
line width = 0.5,
solid},axis x line* = left,x axis line style = {color = {rgb,1:red,0.00000000;green,0.00000000;blue,0.00000000},
draw opacity = 1.0,
line width = 1,
solid},scaled y ticks = false,ylabel style = {font = {\fontsize{11 pt}{14.3 pt}\selectfont}, color = {rgb,1:red,0.00000000;green,0.00000000;blue,0.00000000}, draw opacity = 1.0, rotate = 0.0},log basis y=10,ymajorgrids = true,ytick = {1.0e-15,3.162277660168379e-13,1.0e-10,3.162277660168379e-8,1.0e-5},yticklabels = {$10^{-15.0}$,$10^{-12.5}$,$10^{-10.0}$,$10^{-7.5}$,$10^{-5.0}$},ytick align = inside,yticklabel style = {font = {\fontsize{8 pt}{10.4 pt}\selectfont}, color = {rgb,1:red,0.00000000;green,0.00000000;blue,0.00000000}, draw opacity = 1.0, rotate = 0.0},y grid style = {color = {rgb,1:red,0.00000000;green,0.00000000;blue,0.00000000},
draw opacity = 0.1,
line width = 0.5,
solid},axis y line* = left,y axis line style = {color = {rgb,1:red,0.00000000;green,0.00000000;blue,0.00000000},
draw opacity = 1.0,
line width = 1,
solid},    xshift = 0.0mm,
    yshift = 0.0mm,
    axis background/.style={fill={rgb,1:red,1.00000000;green,1.00000000;blue,1.00000000}}
,legend style = {color = {rgb,1:red,0.00000000;green,0.00000000;blue,0.00000000},
draw opacity = 1.0,
line width = 1,
solid,fill = {rgb,1:red,1.00000000;green,1.00000000;blue,1.00000000},font = {\fontsize{8 pt}{10.4 pt}\selectfont}},colorbar style={title=}, xmode = {log}, ymin = {1.0e-15}, width = \figurewidth]\addplot+ [color = {rgb,1:red,0.00000000;green,0.60560316;blue,0.97868012},
draw opacity = 1.0,
line width = 2,
solid,mark = diamond*,
mark size = 1.5,
mark options = {
    color = {rgb,1:red,0.00000000;green,0.00000000;blue,0.00000000}, draw opacity = 1.0,
    fill = {rgb,1:red,0.00000000;green,0.60560316;blue,0.97868012}, fill opacity = 1.0,
    line width = 1,
    rotate = 0,
    solid
},forget plot]coordinates {
(10.000000000000002, 2.516353892923855e-7)
(16.68100537200059, 5.072349546740066e-7)
(27.825594022071247, 4.702456324655572e-7)
(46.415888336127786, 6.606871992435806e-8)
(77.4263682681127, 2.6086480071106027e-8)
(129.15496650148842, 1.4812809546983616e-8)
(215.44346900318828, 5.438690340315867e-9)
(359.38136638046274, 1.9731604332803265e-9)
(599.4842503189407, 7.045323403432953e-10)
(999.9999999999998, 2.5980106338170857e-10)
};
\addplot+ [color = {rgb,1:red,0.88887350;green,0.43564919;blue,0.27812294},
draw opacity = 1.0,
line width = 2,
solid,mark = +,
mark size = 1.5,
mark options = {
    color = {rgb,1:red,0.00000000;green,0.00000000;blue,0.00000000}, draw opacity = 1.0,
    fill = {rgb,1:red,0.88887350;green,0.43564919;blue,0.27812294}, fill opacity = 1.0,
    line width = 1,
    rotate = 0,
    solid
},forget plot]coordinates {
(10.000000000000002, 1.0822875873271418e-8)
(16.68100537200059, 3.321360362444935e-8)
(27.825594022071247, 2.754906010464148e-8)
(46.415888336127786, 2.6748027149368633e-8)
(77.4263682681127, 7.812672254155632e-8)
(129.15496650148842, 7.299092078602733e-9)
(215.44346900318828, 2.884203706474628e-9)
(359.38136638046274, 1.7234734108276357e-9)
(599.4842503189407, 3.731206323636671e-10)
(999.9999999999998, 1.2569370685392715e-10)
};
\addplot+ [color = {rgb,1:red,0.24222430;green,0.64327509;blue,0.30444865},
draw opacity = 1.0,
line width = 2,
solid,mark = square*,
mark size = 1.5,
mark options = {
    color = {rgb,1:red,0.00000000;green,0.00000000;blue,0.00000000}, draw opacity = 1.0,
    fill = {rgb,1:red,0.24222430;green,0.64327509;blue,0.30444865}, fill opacity = 1.0,
    line width = 1,
    rotate = 0,
    solid
},forget plot]coordinates {
(10.000000000000002, 5.966139454245803e-8)
(16.68100537200059, 2.0409939496439397e-9)
(27.825594022071247, 2.3755556445392624e-9)
(46.415888336127786, 1.1961414104319763e-9)
(77.4263682681127, 4.2053709576532427e-10)
(129.15496650148842, 1.2934062116872877e-10)
(215.44346900318828, 3.715729957802627e-11)
(359.38136638046274, 1.029778208990083e-11)
(599.4842503189407, 2.7844421939513757e-12)
(999.9999999999998, 7.441009345075821e-13)
};
\addplot+ [color = {rgb,1:red,0.76444018;green,0.44411178;blue,0.82429754},
draw opacity = 1.0,
line width = 2,
solid,mark = triangle*,
mark size = 1.5,
mark options = {
    color = {rgb,1:red,0.00000000;green,0.00000000;blue,0.00000000}, draw opacity = 1.0,
    fill = {rgb,1:red,0.76444018;green,0.44411178;blue,0.82429754}, fill opacity = 1.0,
    line width = 1,
    rotate = 0,
    solid
},forget plot]coordinates {
(10.000000000000002, 7.017599160696175e-8)
(16.68100537200059, 9.061488943123742e-9)
(27.825594022071247, 7.441646284953662e-10)
(46.415888336127786, 3.013672421369384e-11)
(77.4263682681127, 2.9448226418633345e-11)
(129.15496650148842, 7.915533132855273e-12)
(215.44346900318828, 1.566889521292974e-12)
(359.38136638046274, 2.8681380911471083e-13)
(599.4842503189407, 1.161880770121101e-13)
(999.9999999999998, 2.0969169130268262e-13)
};
\end{axis}

\end{tikzpicture}

%% file: relerror_chebyshev_n8_v2.tikz
\begin{tikzpicture}[]
\begin{axis}[height = \figureheight, ylabel = {}, xmin = {8.709635899560805}, xmax = {1148.1536214968828}, ymax = {1.0e-5}, ymode = {log}, xlabel = {k}, unbounded coords=jump,scaled x ticks = false,xlabel style = {font = {\fontsize{11 pt}{14.3 pt}\selectfont}, color = {rgb,1:red,0.00000000;green,0.00000000;blue,0.00000000}, draw opacity = 1.0, rotate = 0.0},log basis x=10,xmajorgrids = true,xtick = {10.0,31.622776601683793,100.0,316.22776601683796,1000.0},xticklabels = {$10^{1.0}$,$10^{1.5}$,$10^{2.0}$,$10^{2.5}$,$10^{3.0}$},xtick align = inside,xticklabel style = {font = {\fontsize{8 pt}{10.4 pt}\selectfont}, color = {rgb,1:red,0.00000000;green,0.00000000;blue,0.00000000}, draw opacity = 1.0, rotate = 0.0},x grid style = {color = {rgb,1:red,0.00000000;green,0.00000000;blue,0.00000000},
draw opacity = 0.1,
line width = 0.5,
solid},axis x line* = left,x axis line style = {color = {rgb,1:red,0.00000000;green,0.00000000;blue,0.00000000},
draw opacity = 1.0,
line width = 1,
solid},scaled y ticks = false,ylabel style = {font = {\fontsize{11 pt}{14.3 pt}\selectfont}, color = {rgb,1:red,0.00000000;green,0.00000000;blue,0.00000000}, draw opacity = 1.0, rotate = 0.0},log basis y=10,ymajorgrids = true,ytick = {1.0e-15,3.162277660168379e-13,1.0e-10,3.162277660168379e-8,1.0e-5},yticklabels = {$10^{-15.0}$,$10^{-12.5}$,$10^{-10.0}$,$10^{-7.5}$,$10^{-5.0}$},ytick align = inside,yticklabel style = {font = {\fontsize{8 pt}{10.4 pt}\selectfont}, color = {rgb,1:red,0.00000000;green,0.00000000;blue,0.00000000}, draw opacity = 1.0, rotate = 0.0},y grid style = {color = {rgb,1:red,0.00000000;green,0.00000000;blue,0.00000000},
draw opacity = 0.1,
line width = 0.5,
solid},axis y line* = left,y axis line style = {color = {rgb,1:red,0.00000000;green,0.00000000;blue,0.00000000},
draw opacity = 1.0,
line width = 1,
solid},    xshift = 0.0mm,
    yshift = 0.0mm,
    axis background/.style={fill={rgb,1:red,1.00000000;green,1.00000000;blue,1.00000000}}
,legend style = {color = {rgb,1:red,0.00000000;green,0.00000000;blue,0.00000000},
draw opacity = 1.0,
line width = 1,
solid,fill = {rgb,1:red,1.00000000;green,1.00000000;blue,1.00000000},font = {\fontsize{8 pt}{10.4 pt}\selectfont}},colorbar style={title=}, xmode = {log}, ymin = {1.0e-15}, width = \figurewidth]\addplot+ [color = {rgb,1:red,0.00000000;green,0.60560316;blue,0.97868012},
draw opacity = 1.0,
line width = 2,
solid,mark = diamond*,
mark size = 1.5,
mark options = {
    color = {rgb,1:red,0.00000000;green,0.00000000;blue,0.00000000}, draw opacity = 1.0,
    fill = {rgb,1:red,0.00000000;green,0.60560316;blue,0.97868012}, fill opacity = 1.0,
    line width = 1,
    rotate = 0,
    solid
}]coordinates {
(10.000000000000002, 5.085566270533628e-12)
(16.68100537200059, 1.5370360937454182e-12)
(27.825594022071247, 4.091597906291278e-14)
(46.415888336127786, 1.4028183848321857e-13)
(77.4263682681127, 1.952663890667564e-14)
(129.15496650148842, 1.1491099584086035e-14)
(215.44346900318828, 1.5786070872004991e-13)
(359.38136638046274, 4.117468984258728e-13)
(599.4842503189407, 5.372780076291311e-13)
(999.9999999999998, 4.689044673813537e-12)
};
\addlegendentry{m=0.8}
\addplot+ [color = {rgb,1:red,0.88887350;green,0.43564919;blue,0.27812294},
draw opacity = 1.0,
line width = 2,
solid,mark = +,
mark size = 1.5,
mark options = {
    color = {rgb,1:red,0.00000000;green,0.00000000;blue,0.00000000}, draw opacity = 1.0,
    fill = {rgb,1:red,0.88887350;green,0.43564919;blue,0.27812294}, fill opacity = 1.0,
    line width = 1,
    rotate = 0,
    solid
}]coordinates {
(10.000000000000002, 9.404832632851405e-13)
(16.68100537200059, 1.9276395063521136e-13)
(27.825594022071247, 3.227584408609967e-14)
(46.415888336127786, 2.8638459940876785e-14)
(77.4263682681127, 1.0992539641257146e-13)
(129.15496650148842, 2.492863838613256e-14)
(215.44346900318828, 1.1799115066073282e-13)
(359.38136638046274, 4.674011393303684e-13)
(599.4842503189407, 1.5173471623593503e-13)
(999.9999999999998, 6.967083218007258e-13)
};
\addlegendentry{m=0.9}
\addplot+ [color = {rgb,1:red,0.24222430;green,0.64327509;blue,0.30444865},
draw opacity = 1.0,
line width = 2,
solid,mark = square*,
mark size = 1.5,
mark options = {
    color = {rgb,1:red,0.00000000;green,0.00000000;blue,0.00000000}, draw opacity = 1.0,
    fill = {rgb,1:red,0.24222430;green,0.64327509;blue,0.30444865}, fill opacity = 1.0,
    line width = 1,
    rotate = 0,
    solid
}]coordinates {
(10.000000000000002, 3.939934570590639e-14)
(16.68100537200059, 4.932454750797928e-14)
(27.825594022071247, 1.0415860763917689e-14)
(46.415888336127786, 8.194426879352045e-15)
(77.4263682681127, 7.457977676679274e-15)
(129.15496650148842, 1.633636503787527e-14)
(215.44346900318828, 1.769022747089059e-14)
(359.38136638046274, 7.290343415701359e-15)
(599.4842503189407, 5.349735810568049e-14)
(999.9999999999998, 1.0349633858017259e-14)
};
\addlegendentry{m=1.0}
\addplot+ [color = {rgb,1:red,0.76444018;green,0.44411178;blue,0.82429754},
draw opacity = 1.0,
line width = 2,
solid,mark = triangle*,
mark size = 1.5,
mark options = {
    color = {rgb,1:red,0.00000000;green,0.00000000;blue,0.00000000}, draw opacity = 1.0,
    fill = {rgb,1:red,0.76444018;green,0.44411178;blue,0.82429754}, fill opacity = 1.0,
    line width = 1,
    rotate = 0,
    solid
}]coordinates {
(10.000000000000002, 2.3440804773804745e-13)
(16.68100537200059, 8.895018605290465e-15)
(27.825594022071247, 1.7932166462874113e-15)
(46.415888336127786, 5.5683303065201865e-15)
(77.4263682681127, 6.650870074384422e-15)
(129.15496650148842, 3.9617017867790086e-14)
(215.44346900318828, 1.8009148890781935e-14)
(359.38136638046274, 7.967102378815545e-14)
(599.4842503189407, 1.0413671442468686e-13)
(999.9999999999998, 7.769551656829633e-14)
};
\addlegendentry{m=1.1}
\end{axis}

\end{tikzpicture}